\documentclass[11pt]{article} 

\usepackage{amsmath, amsfonts, amssymb}
\usepackage{mathrsfs}
\usepackage{theorem}
\usepackage{bm}
\usepackage[dvipdf]{graphicx}

\RequirePackage[colorlinks,citecolor=blue,urlcolor=blue,bookmarksopen]{hyperref}

\newtheorem{mythm}{Theorem}[section]
\newtheorem{myprop}[mythm]{Proposition}
\newtheorem{mylem}[mythm]{Lemma}
\newtheorem{mycor}[mythm]{Corollary}
\newtheorem{mydefn}[mythm]{Definition}

{\newtheorem{myrem}[mythm]{Remark}}
{
\newtheorem{myexam}[mythm]{Example}}%

\def\R{\mathbb R}

\def\C{\mathscr C}

\def\F{\mathscr F}

\def\E{\mathbb E}
\def\p{\mathbb P}
\def\e{\text{\rm{e}}}

\def\la{\langle}\def\d{\text{\rm{d}}}
\def\raa{\rangle}

\def\veps{\varepsilon}

\def\law{\mathscr{L}}

\def\pb{\mathscr{P}}

\def\wt{\widetilde}

\def\var{\mathrm{var}}
\def\W{\mathbb{W}}

\newcommand{\bard}{\bar{\mathscr{D}}_1^r}

\newenvironment{proof}{{\noindent\it Proof.}\ }{\hfill $\square$\par}

\numberwithin{equation}{section}

\allowdisplaybreaks

\begin{document}

\title{Viscosity solutions to HJB equations associated with optimal control problem  for McKean-Vlasov SDEs\footnote{Supported in part by National Key R\&D Program of China (No. 2022YFA1006000) and NNSFs of China (No. 12271397,  11831014)}}

\author{Jinghai Shao\thanks{Center for Applied Mathematics, Tianjin University, Tianjin 300072, China. Email: shaojh@tju.edu.cn
}}
\date{}
\maketitle

\begin{abstract}
  This work concerns the optimal control problem for McKean-Vlasov SDEs.  In order to characterize the value function, we develop the viscosity solution theory for Hamilton-Jacobi-Bellman (HJB) equations on the Wasserstein space using Mortensen's derivative. In particular,  a comparison principle for viscosity solution is established. Our approach is based on Borwein-Preiss variational principle to overcome the loss of compactness for bounded sets in the Wasserstein space.
\end{abstract}



\section{Introduction}

This paper investigates viscosity solution theory to HJB equations associated with the optimal feedback control problem for McKean-Vlasov SDEs. Consider the following McKean-Vlasov SDE:
\begin{equation}\label{o-1}
\d X_t=b(t,X_t,\law_{X_t},\alpha_t)\d t+\sigma(t,X_t,\law_{X_t})\d W_t,\quad X_s=\xi\ \text{with $\law_{\xi}=\mu$},
\end{equation}
where $b:[0,T]\times \R^d\times \pb(\R^d)\times \pb(U)\to \R^d$, $\sigma:[0,T]\times \R^d\times \pb(\R^d)\to \R^{d\times d}$; $(W_t)_{t\geq 0}$ is a $d$-dimension Wiener process; $\pb(\R^d)$ ($\pb(U)$) denotes the collection of all probability measures over $\R^d$ ($U$ respectively), called Wasserstein space for simplicity; $U$ is a compact set in $\R^k$ for some positive integer $k$; $\law_{\xi}$ denotes the distribution of the random variable $\xi$; $(\alpha_t)$   represents the control policy belong to the set of admissible controls $\Pi_{s,\mu}$, detailed in next section. Given certain measurable functions $f$ and $g$, we aim to minimize the following objective function: for $0\leq s<T<\infty$,
\begin{equation}\label{o-2} J(s,\mu;\alpha):=\E\Big[\int_s^Tf(r,X_r,\law_{X_r},\alpha_r)\d r+g(X_T,\law_{X_T})\Big].
\end{equation}
The corresponding value function is given by
\begin{equation}\label{o-3}
V(s,\mu)=\inf\nolimits_{\alpha\in \Pi_{s,\mu}} J(s,\mu;\alpha).
\end{equation}

The optimal control problem of McKean-Vlasov SDEs is motivated by the mean-field game theory developed by \cite{LL06} and \cite{HCM06}. It has been studied in \cite{AD10,BDL,CD15} by maximum principle method, in \cite{BFY15,BIMS,LP16,Pham17} by dynamic programming method.
As the state variable of the value function contains the probability measure, the optimal control problem \eqref{o-3} is essentially an infinite dimensional problem. To characterize the value function, many approaches have been proposed to develop  the viscosity solution theory to HJB equations on the Wasserstein space.  For instance, Burzoni et al. \cite{BIMS} used the linear functional derivative; Gangbo et al. \cite{GNT} used the Riemannian tangent space structure; Bensoussan et al. \cite{BFY15} studied the probability densities, which were viewed as elements in $L^2(\R^d)$ and develop the G\^ateaux differential structure in the Hilbert space $L^2(\R^d)$.  See \cite[Chapter 5]{Car1} for the discussion on the relationship between different derivatives in the Wasserstein space. In order to guarantee the value function to be the unique viscosity solution of the corresponding HJB equation, the crucial point is to establish the comparison principle for viscosity sub- and super-solution. On this topic, Pham and Wei \cite{Pham17} adopted the approach of Lions' lifting  \cite{Card,Lions}, which provides  a lifting identification between measures and random variables, then used the theory on viscosity solutions to HJB equations in Hilbert space (cf. \cite{Li88}) to characterize the value function. This method works in the Hilbert space $L^2(\F_0;\R^d)$ instead of the Wasserstein space itself.
Recently, Burzoni et al. \cite{BIMS} developed an intrinsic approach to prove the uniqueness of viscosity solution without using Lions' lifting. In \cite{BIMS}, the linear functional derivative on $\pb(\R^d)$ is used and the key point is the subtle construction of a distance-like function on the Wasserstein space, that is,   for any two probability measures $\mu$, $\nu$  in $\pb(\R^d)$,
\[d(\mu,\nu):=\sum_{j=1}^\infty c_j\la \mu-\nu, f_j\raa^2,\]
where the countable set $\{f_j\}_{j\geq 1}$ is carefully constructed  in order to estimate the linear functional derivative of $d$ by itself. Besides, the set of admissible controls in \cite{BIMS} is restricted to contain only the measurable deterministic functions of time.

In our previous work \cite{Sh23a},  the existence of the optimal Markovian feedback controls and the dynamic programming principle have been investigated.
Our purpose of current work is to develop the viscosity solution theory to HJB equations on the Wasserstein space to characterize the value function. To this aim,  we shall use Mortensen's derivative on the Wasserstein space to establish the HJB equations satisfied by the value function.  When the diffusion coefficient $\sigma$ in \eqref{o-1} is a constant nondegenerate matrix, the distribution of the controlled process $X_t$ owns probability density w.r.t. the Lebesgue measure. Mortensen \cite{Mor66a,Mor66} introduced the Fr\'echet derivative for the functionals of the probability densities, and   Benes and Karatzas \cite{BK83} developed it to solve nonlinear filtering problem. We also use Mortensen's derivative (with certain modification), see Definition \ref{def-6} below,  to characterize the value function as a viscosity solution to a differential equation on the space of probability densities.

The advantage of adopting Mortensen's derivative consists of two folds: first, the tangent spaces at every probability measure are the same, so the Wasserstein space with such a differential structure can be viewed as a flat space. In contrast, the Wasserstein space studied such as in Gangbo and Swiech \cite{GS14}, Gangbo et al.\,\cite{GNT} and  Shao \cite{Sh23c}, is a curved space. Clearly, it is easier to do calculus in the Wasserstein space using Mortensen's derivative. Second,     we can find suitable smooth functions w.r.t. Mortensen's differential geometry to establish the comparison principle on the Wasserstein space similar to the standard investigation of HJB equations on the Euclidean space (cf. e.g. \cite{CL,FL93,YZ99}).  Nevertheless, since the Wasserstein space is an infinite dimensional space, and bounded sets on it are not necessary precompact, we use Borwein-Preiss variational principle to ensure the existence of maximum/minimum points in the study of viscosity sub/super-solution.

The theory of viscosity solutions for Hamilton-Jacobi equations in infinite dimensional spaces was initiated by Crandal and Lions \cite{CL85}, where Banach space with Radon-Nikodym property was considered. Ambrosio and Feng \cite{AF14} studied Hamilton-Jacobi equations in geodesic metric spaces, where metric derivative was used, and the results were applied to the Wasserstein space $\pb_2(\R^d)$.  Gangbo and Swiech \cite{GS15} also considered the Hamilton-Jacobi equations using the metric derivative structure on the Wasserstein space.

This work is organized as follows. In Section 2, we present the framework of the optimal control problem, especially introduce the set of control policies studied in this work.  In Section 3, a subspace $\pb_1^r(\R^d)$ is introduced, which is proved to be conservative for the controlled process; see Lemma \ref{lem-7} below. Then, under certain explicit conditions,  the value function is shown to be a  viscosity solution to an appropriate HJB equation on $\pb_1^r(\R^d)$ in terms of Mortensen's derivative. Section 4 is devoted to establishing the comparison principle for the HJB equation, which ensures that the value function is the unique viscosity solution satisfying certain continuity property.

\section{Framework of the optimal control problem}

In a given probability space $(\Omega,\F,\p)$, we use $\law_X$ to denote the law of a random variable $X$. For two probability measures $\mu,\nu$ over $\R^d$, the total variation distance between them is defined by
\[\|\mu-\nu\|_\var=\sup\big\{|\la f,\mu \raa -\la f, \nu\raa|;\,f\in \mathscr{B}(\R^d), |f|\leq 1\big\},
\]where $\la f,\mu\raa=\int_{\R^d} f(x)\mu(\d x)$.
The $L^p$-Wasserstein distance between $\mu$ and $\nu$ is defined by
\begin{equation}\label{W-1}
\W_p(\mu,\nu)=\inf_{\Gamma\in \mathcal{C}(\mu,\nu)} \!\Big(\int_{\R^d\times\R^d}\!\!|x-y|^p\,\Gamma(\d x,\d y)\Big)^{\frac 1p},\quad p\geq 1,
\end{equation}
where $\mathcal{C}(\mu,\nu)$ denotes the collection of all couplings of $\mu$, $\nu$ on $\R^d\times\R^d$.
Let $T>0$ be a given constant throughout this work. $U$ is a compact subset of $\R^k$ for some $k\geq 1$. Let $\pb(U)$ and $\pb(\R^d)$ denote the set of probability measures over $U$ and  $\R^d$ respectively. For $p\geq 1$, put
\begin{align*}\pb_p(\R^d)&=\Big\{\mu\in \pb(\R^d);\, \int_{\R^d}|x|^p\mu(\d x)<\infty\Big\}.
\end{align*}
Let $\C([s,t];\R^d)$ be the path space of all continuous functions from $[s,t]$ to $\R^d$.

For a measurable function $u:U\to \R$, it can be generalized as a functional on $\pb(U)$ via
\[u(\alpha):=\int_{U}u(x)\alpha(\d x),\quad \forall \, \alpha\in \pb(U).\] In the following, we often use $u(\alpha)$ instead of $\la u,\alpha\raa$ or $\int_{U}u(x)\alpha(\d x)$ to simplify the notation without mentioning it again.

Consider the following controlled stochastic dynamical system characterized by a SDE of McKean-Vlasov type:
\begin{equation}\label{a-1}
\d X_t= b(t,X_t,\law_{X_t},\alpha_t )  \d t+\sigma(t,X_t,\law_{X_t})\d W_t,\quad X_s=\xi,
\end{equation}
where $b:[0,T]\times \R^d\times \pb(\R^d)\times \pb(U)\to \R^d$, $\sigma:[0,T]\times \R^d\times \pb(\R^d)\to \R^{d\times d}$, $(W_t)_{t\geq 0}$ is a $d$-dimensional Brownian motion; $(\alpha_t)$ is a $\pb(U)$-valued process represented the controls  imposed  on the studied system.

\begin{mydefn}\label{def-1}
Given $s\in [0,T)$, $\mu\in \pb_1(\R^d)$, a Markovian feedback control is a term $\Theta=(\Omega,\F,\p$, $ \{\F_t\}_{t\geq 0}, W_\cdot, X_\cdot, \alpha_\cdot)$ such that
\begin{enumerate}
  \item[$\mathrm{(i)}$] $(\Omega,\F,\p)$ is a probability space with filtration $\{\F_t\}_{t\geq 0}$;
  \item[$\mathrm{(ii)}$] $(W_t)_{t\geq 0}$ is an $\F_t$-adapted Brownian motion;
  \item[$\mathrm{(iii)}$] $(\alpha_t)_{t\in [s,T]}$ is a $\pb(U)$-valued measurable process on $\Omega$ adapted to $\{\F_t\}$  such that for any $\F_s$-measurable random variable $\xi$ satisfying $\law_\xi=\mu$, SDE \eqref{a-1} admits a weak solution $(X_t)_{t\in [s,T]}$  with initial value $X_s=\xi$, and the uniqueness in law holds for SDE \eqref{a-1}.
  \item[$\mathrm{(iv)}$] For each $t\in [s,T]$, there exists a measurable functional $F_t\R^d\to \pb(U)$ such that $\alpha_t=F_t(X_t)$.
\end{enumerate}
\end{mydefn}
The set of all feedback controls $\Theta$ corresponding to the initial value $(s,\mu)$ is denoted by $\Pi_{s,\mu}$.

Given two measurable functions $f:[0,T]\times \R^d\times \pb(\R^d)\times \pb(U)\to \R$ and $g:\R^d\times \pb(\R^d)\to \R$, the objective function is defined by
\begin{equation}\label{a-3}
J(s,\mu;\Theta)=\E_{s,\mu}\Big[\int_s^{T } f(r,X_r,\law_{X_r},\alpha_r)\d r+g(X_{T},\law_{X_{T}})\Big]
\end{equation} for $s\!\in \![0,T)$ and $\mu\!\in\! \pb(\R^d)$, where $\E_{s,\mu}$ means taking expectation w.r.t.\,the initial value $\law_{X_s}\!=\!\mu$.
The value function is defined by
\begin{equation}\label{a-4}
V(s,\mu)=\inf_{\Theta\in \Pi_{s,\mu}} J(s,\mu;\Theta).
\end{equation}
A feedback control $\Theta^\ast\in \Pi_{s,\mu}$ is said to be optimal, if it satisfies
 $J(s,\mu;\Theta^\ast)=V(s,\mu).$

Notice that under the uniqueness condition of Definition \ref{def-1}(iii), $J(s,\mu;\Theta)$ and hence $V(s,\mu)$ are well defined.
Namely, for any two given  random variables $\xi,\tilde \xi$ with $\law_\xi\!=\!\law_{\tilde \xi}\!\in \!\pb(\R^d)$,  the values of $J(s,\mu;\Theta)$ and $V(s,\mu)$ will be fixed no matter which initial value $X_s=\xi$ or $X_s=\tilde \xi$ is used to determine the solution of SDE \eqref{a-1} for the process $(X_t)$ in \eqref{a-3}.
Indeed,
noticing \[\E[g(X_T,\law_{X_T})]=\int_{\R^d} g(x,\law_{X_T}) \law_{X_T}(\d x),\] we see that the term $\E[g(X_T,\law_{X_T})]$ in \eqref{a-3} depends only on the distribution of $X_T$, which is determined by the law $\mu$ of $\xi$ due to Definition \ref{def-1}(iii). Similar treatment can be applied to the term $\E\Big[\int_s^T \!f(r,X_r,\law_{X_r}, \alpha_r)\d r\Big]$ for the feedback control $\alpha_t=F_t(X_t)$.

We introduce the conditions on the controlled system used below.
\begin{itemize}
  \item[$\mathrm{(H1)}$] There exists a constant $K_1$ such that
      \[|b(t,x,\nu_1, \alpha)-b(s,y,\nu_2, \alpha)|
      +\|\sigma(t,x,\nu_1)-\sigma(s,y,\nu_2)\|\leq K_1\big(|t-s|+|x-y|+\W_1(\nu_1,\nu_2)\big)\]
      for all $s,t\in [0,T]$, $x,y\in \R^d$, $\nu_1,\nu_2\in \pb_1(\R^d)$,   and $\alpha\in \pb(U)$.
  \item[$\mathrm{(H2)}$] There exists a constant $K_2$ such that
      \[|b(t,x,\mu, \alpha)|\!+\!\|\sigma(t,x,\mu)\|
      \leq K_2\Big(1\!+\!|x|\!+\!\int_{\R^d}\! |z|\,\mu(\d z)\Big)
       \]  for $t\in [0,T], x\in \R^d,   \mu\in \pb_1(\R^d), \alpha\in \pb(U).$
  \item[$\mathrm{(H3)}$] There exist constants $K_3$ and $K_4$ such that
      \begin{gather*}
      |f(t,x,\mu,\alpha)-f(s,y,\nu,\alpha)|+ |g(x,\mu)-g(y,\nu)| \leq K_3\big(|t-s|+|x-y|+\W_1(\mu,\nu)\big),\\
      |f(t,x,\mu,\alpha)|+|g(x,\mu)|\leq K_4\Big(1+|x|+\int_{\R^d}|z|\mu(\d z)\Big)
      \end{gather*}
      for $t,s\in[0,T]$, $x,y\in \R^d$, $\mu,\nu\in \pb_1(\R^d)$, and $\alpha\in \pb(U)$.
\end{itemize}

The conditions (H1) and (H2) are used to ensure the existence of solution to SDE \eqref{a-1} under suitable control process $(\alpha_t)$. For example, under (H1) and (H2), for each $\alpha\in \pb(U)$, we consider the control strategy $\alpha_t\equiv \alpha\in \pb(U)$, then SDE \eqref{a-1} admits a unique weak solution for any initial value $X_s=\xi$ with $\law_\xi\in \pb_1(\R^d)$; see, e.g.
\cite[Theorems 2.1 and 3.1]{Fun}. Moreover, if assume, in addition, $\sigma(t,x,\mu)$ depends only on $t,x$, then SDE \eqref{a-1} admits a unique strong solution for $\alpha_t\equiv \alpha$ due to \cite[Theorem 2.1]{Wang18}.  So, there is an admissible control $\Theta$ associated with such deterministic control strategy $\alpha_t\equiv \alpha$.
Consequently, $\Pi_{s,\mu}$ is not empty.

According to \cite{Sh23a}, the dynamic programming principle holds for $V$ given in \eqref{a-4}.

\begin{myprop}[\cite{Sh23a}, Proposition 4.1]\label{lem-3}
Suppose that $\mathrm{(H1)}$ and $\mathrm{(H2)}$ hold. Then, for any $s\leq t\leq T$, $\mu\in \pb_1(\R^d)$,  it holds
\begin{equation}\label{b-1}
V(s,\mu)=\inf_{\Theta\in \Pi_{s,\mu}}\Big\{ \E_{s,\mu}\Big[\int_s^t\!f(r,X^{s,\mu,\alpha}_r, \law_{X^{s,\mu,\alpha}_r},\alpha_r)\d r+V(t, \law_{X^{s,\mu,\alpha}_t})\Big]\Big\},
\end{equation} where  $X_\cdot^{s,\mu,\alpha}$ stands for  the controlled process associated with $\Theta\in \Pi_{s,\mu}$.
\end{myprop}

\begin{myprop}[\cite{Sh23a}, Proposition 4.3]\label{lem-2}
Assume that   $\mathrm{(H1)}$-$\mathrm{(H3)}$ hold and $\sigma(t,x,\mu)$ depends only on $t,x$, then the value function  $V(s,\mu)$   defined in \eqref{a-4} satisfies that there exists a constant $C>0$ such that
\begin{equation}\label{cont}
|V(s,\mu)-V(s',\mu')|\leq C\big( \sqrt{|s'-s|} +\W_1(\mu,\mu')\big)
\end{equation} for any $s,s'\in [0,T]$, $\mu,\mu'\in \pb_1(\R^d)$.
\end{myprop}

\section{Existence of viscosity solutions to HJB equations}

The value functions are usually not smooth enough, and hence it is necessary to develop viscosity solution theory to the HJB equations on the Wasserstein space similar to the classical setting in $\R^d$ (cf. e.g. \cite{CL}). Now the value function $V$ acts on the Wasserstein space  $\pb(\R^d)$, on which the derivative of $\mu\mapsto V(s,\mu)$ can be defined in several different ways. As the Wasserstein space is an infinite dimensional space, these derivatives  on $\pb(\R^d)$ are quite different to each other. To choose which derivative on $\pb(\R^d)$ to establish the  HJB equation, from the viewpoint of application, a primary request is to ensure that the value function is the unique viscosity solution to the corresponding HJB equation so that numerical algorithm can be used to approximate the value function.

We shall introduce a subspace $\pb^r_1(\R^d)$ in \eqref{f-0} of $\pb_1(\R^d)$ and show that $\pb^r_1(\R^d)$  is conservative for the controlled process $(X_t)$ in Lemma \ref{lem-7}. Then, the value function is proved to be a viscosity solution to an appropriate HJB equation on $\pb^r_1(\R^d)$ in Theorem \ref{thm-3}. As being well-known, the main challenge to characterize the value function is to establish the uniqueness of viscosity solution to the corresponding HJB equation. To this aim we introduce an auxiliary function $\gamma$ in \eqref{norm} and a weighted Sobolev norm associated with $\gamma$.   Also, we put forward the square of weighted Sobolev norm $F$ on $\pb^r_1(\R^d)$ as an auxiliary functional in Lemma \ref{lem-8} and study its Mortensen's derivative. Then, using the Borwein-Preiss variational principle to overcome the loss of compactness for bounded sets in the Wasserstein space, we finally establish the comparison principle in Theorem \ref{thm-4}, which implies the uniqueness of viscosity solution as desired.

Let us  first introduce a   subspace of $\pb(\R^d)$:
\begin{equation} \label{f-0}
\pb_1^r(\R^d) =\Big\{\mu\!\in\! \pb_1(\R^d);\, \text{$\rho=\!\frac{\d \mu }{\d x}\!\in \!\mathcal{C}^1(\R^d) $  satisfying $\int_{\R^d}\!  \e^{|x|}\big(\rho(x)^2\!+\!|D \rho(x)|^2\big)\d x<\infty$}\Big\},
\end{equation}
where $\frac{\d \mu }{  \d x}$ denotes the Radon-Nikodym derivative of $\mu$  w.r.t.\,the Lebesgue measure $\d x$.
Let
\begin{equation}\label{f-0.5}
\mathscr{D}_1^r=\big\{\rho\in \mathcal{C}^1(\R^d); \rho(x)\d x\in \pb_1^r(\R^d)\big\}
\end{equation} be the collection of densities of the probability measures in $\pb_1^r(\R^d)$. In the sequel, in order to simplify the notation   we identify the functional
$ u$ over $\pb_1^r(\R^d)$ with a functional over the set of the probability densities $\mathscr{D}_1^r$, that is,   let $u(\rho)=u(\mu)$ if $\rho=\frac{\d \mu}{\d x}$ for   $\mu \in \pb_1^r(\R^d)$.

For a nonnegative integer $k$ and a domain $B\subset \R^d$, let $\mathcal{C}^k(B)$ denote the vector space of all functions $\phi$ which, together with all their partial derivatives $D^\alpha \phi$ of orders $|\alpha|\leq k$, are continuous on $B$.

In the study of the uniqueness of viscosity solution to HJB equations, it is known that the continuity of the solution plays important role. By Proposition \ref{lem-2},  we can show the value function $V(s,\mu)$ is continuous in $\mu$ w.r.t.\! the $L^1$-Wasserstein distance. Nevertheless, if we view $V$ as a function on $\mathscr{D}_1^r$, we need to introduce a new distance on $\mathscr{D}_1^r$ such that the value function $V$ remains continuous under this distance. To this end, we introduce a new reference function $\gamma(x)$ and a weighted Sobolev norm on $\mathscr{D}_1^r$.

Let $\gamma\in \mathcal{C}^2(\R^d)$ satisfy $\gamma\geq 1$ and
\begin{equation}\label{norm}
  \gamma(x)=\begin{cases} 1, & \text{if $|x|\leq 1$},\\
  \e^{|x|},&\text{if $|x|>2$}.
  \end{cases}
\end{equation}
Then, for $|x|>2$,
\[\partial_{x_i} \gamma(x)=\e^{|x|}\frac{x_i}{|x|},\quad \partial_{x_ix_j}^2\gamma (x)=\e^{|x|}\Big(\frac{x_ix_j}{|x|^2}+\frac{\delta_{ij}}{|x|} -\frac{x_ix_j}{|x|^3}\Big),\quad i,j=1,\ldots, d,\]
where $\delta_{ij}=1$ if $i=j$; $\delta_{ij}=0$, otherwise.
Hence, together with the fact $\gamma \in \mathcal{C}^2(\R^d)$ and $\gamma\geq 1$, there exists a constant   $\kappa\geq 1$ such that
\begin{equation}\label{norm-1}
\begin{split}
  |D \gamma(x)|\leq \kappa \gamma(x),\quad \sum_{i,j=1}^d |\partial_{x_ix_j}^2 \gamma(x)|\leq \kappa\gamma(x),\quad x\in \R^d.
\end{split}
\end{equation}
Let
\[ H_1^2(\R^d)=\big\{h\in L^2(\R^d);   h, D  h\in L^2(\R^d)\big\}.\]
On $L^2(\R^d)$ and $H_1^2(\R^d)$, we introduce a weighted norm
\[\|h\|_{L^2(\gamma)}:=\Big(\int_{\R^d}\! |h(x)|^2\gamma(x)\d x\Big)^{\frac 12},\quad h\in L^2(\R^d), \] and a weighted Sobolev norm
\[\|h\|_{H_1^2(\gamma)} :=\Big(\int_{\R^d}\big( |h(x)|^2+ |Dh(x)|^2\big)\gamma(x)\d x\Big)^{\frac 12},\quad  h\in H_1^2(\R^d).\]
According to \cite[Theorem 6.15]{Vil},
\[\W_1(\rho(x)\d x,\tilde \rho(x)\d x)\leq \int_{\R^d} |x| |\rho(x)-\tilde \rho(x)|\d x,\quad \forall\, \rho(x)\d x,\, \tilde \rho(x)\d x\in \pb_1(\R^d).\]
By H\"older's inequality,
\begin{equation}\label{f-5.0}
\begin{aligned}
  \W_1(\rho(x)\d x,\tilde \rho(x)\d x)&\leq \Big(\int_{\R^d}|x|\gamma(x)^{-1}\d x\Big)^{\frac 12}\Big( \int_{\R^d} |\rho(x)-\tilde \rho(x)|^2\gamma(x)\d x\Big)^{\frac 12}\\
  &=\kappa_4 \|\rho-\tilde \rho\|_{L^2(\gamma)},
\end{aligned}
\end{equation} where   $\kappa_4=\int_{\R^d}|x|\gamma(x)^{-1}\d x<\infty $ due to the definition of   function $\gamma$.

\begin{mydefn}[Mortensen's derivative]\label{def-6}
Let $S$ be a continuous, real-valued function defined on some subset of $H_1^2(\R^d)$. Given a point $\rho$ in whose neighborhood $S$ is well defined. Suppose that there exists a real-valued linear functional $A$ on $H_1^2(\R^d)$ such that for all sequence $\{\phi\}\subset H_1^2(\R^d)$ with $\|\phi\|_{H_1^2(\gamma)}\to 0$,
\begin{equation}\label{f-1}
\frac{|S(\rho+\phi)-S(\rho)-A\phi|}{\|\phi\|_{H_1^2(\gamma)}} \longrightarrow 0,\qquad \text{as}\ \|\phi\|_{H_1^2(\gamma)}\to 0.
\end{equation}
Then, define the Fr\'echet differential of $S$ at $\rho$ in the direction $\phi$ by
\begin{equation}\label{f-2}
\delta S(\rho,\phi)=A \phi.
\end{equation}
Suppose further that $A$ is a continuous linear functional on $H_1^2(\R^d)$, then $A$ is a member of the dual space of $H_1^2(\R^d)$, which owns a representation of the form
\begin{equation}\label{f-3}
A\phi=\int_{\R^d}\big(F(x)\phi(x)+ \la G(x), D\phi(x)\raa \big) \d x.
\end{equation} Define the Fr\'echet derivative of $S$ at $\rho$ by
\begin{equation}\label{f-4}
\Big(\frac{\delta S(\rho)}{\delta \rho(x)},\frac{\delta S(\rho)}{\delta \rho'(x)}\Big)=(F(x), G(x))\in \R\times \R^d,
\end{equation} which is called Mortensen's derivative in this work.
\end{mydefn}
\begin{myrem}\label{rem-5.2}
  The integration in \eqref{f-3} at every point $\rho\in \mathscr{D}_1^r$ is  with respect to the same reference measure,  Lebesgue measure $\d x$, and is independent of $\rho$. From the viewpoint of Riemannian manifolds,   the tangent spaces  associated with Mortensen's derivative are the same for all  $\rho(x)\d x\in \pb^r_1(\R^d)$, and hence we look on  $\pb^r_1(\R^d)$ as a flat space when endowing $\pb^r_1(\R^d)$  with Mortensen's derivative.
\end{myrem}

\begin{myexam}\label{ex-2}
\begin{enumerate}
  \item Suppose $S(\rho)=\int_{\R} k(x)\rho(x)\d x$, then
  $  \big(\frac{\delta S(\rho)}{\delta \rho(x)},\frac{\delta S(\rho)}{\delta \rho'(x)}\big)=(k(x) ,0)$.
  \item Suppose $S(\rho)=\int_{\R} H(\rho(x),\rho'(x))\d x$, where $\rho'(x)=\frac{\d\rho(x)}{\d x}$. Then
      \[\Big(\frac{\delta S(\rho)}{\delta \rho(x)},\frac{\delta S(\rho)}{\delta \rho'(x)}\Big)=\Big( \frac{\partial H(\rho(x),\rho'(x))}{ \partial \rho},   \frac{\partial H(\rho(x),\rho'(x))}{\partial \rho'}\Big). \]
\end{enumerate}
\end{myexam}

Now we are ready to introduce the conditions on the coefficients viewed as functions on $\mathscr{D}_1^r$ instead of $\pb_1^r(\R^d)$. The Lipschitz conditions of $b$, $f$ and $g$ are a litter weaker than those given in conditions $(\mathrm{H}1)$, $(\mathrm{H}3)$ due to \eqref{f-5.0}.
\begin{itemize}
  \item[$(\mathbf{A}_1)$] Suppose that there exist constants $\tilde K_1,\tilde K_2$ such that
\begin{align}\label{f-5.1}
   |b(t,x,\rho,\alpha)\!-\!b(s,y,\tilde \rho,\alpha)|&\leq \tilde K_1 \!\big(|t-s|+\!|x-y|\!+\! \|\rho-\tilde \rho\|_{L^2(\gamma)} \big),\\
|b(t,x,\rho,\alpha)|&\leq \tilde K_2,  \label{f-5.3}
\end{align} for $s,t\in [0,T]$, $x,y\in\R^d$, $\rho,\,\tilde \rho\in \mathscr{D}_1^r$, $\alpha\in\pb(U)$.

  \item[$(\mathbf{A}_2)$] $\sigma(t,x,\rho)\equiv \sigma$, where $\sigma$ is a constant $d\times d$  matrix satisfying $A:=\sigma\sigma^\ast>0$.

  \item[$(\mathbf{A}_3)$] There exists a constant $\tilde K_3$ such that
      \begin{equation}\label{f-5.4}
      \begin{split}
        |f(t,x,\rho,\alpha)\!-\!f(s,y,\tilde \rho,\alpha)|\!+\! |g(x,\rho)\!-\!g(y,\tilde \rho)| &\leq\! \tilde K_3\big(|t\!-\!s|\!+\!|x\!-\!y|\!+ \! \|\rho-\tilde \rho\|_{L^2(\gamma)} \big)\\
        |f(t,x,\rho,\alpha)|\!+ \!|g(x,\rho)|&\leq \tilde K_3,
      \end{split}
      \end{equation}
      for $s,t\in [0,T]$, $x,y\in\R^d$, $\rho,\,\tilde\rho\in \mathscr{D}_1^r$, $\alpha\in\pb(U)$.
\end{itemize}

Let $(X_t)_{t\in [s,T]}$ be a controlled stochastic process satisfying \eqref{a-1} associated with the Markovian feedback control $\Theta\in \Pi_{s,\mu_0}$ with initial value $\law_{X_s}=\mu_0\in \pb_1^r(\R^d)$.
Under the nondegenerate condition $(\mathbf{A}_2)$,   the distribution of $X_t$  admits a density $\rho_t $. Moreover, $\rho_t $ satisfies the nonlinear Fokker-Planck-Kolmogorov equation:
\begin{equation}\label{f-5}
\begin{split}
\partial_t\rho_t(x)&=
\mathcal{L}_\alpha^\ast \rho_t(x)\\
&:=\frac 12\sum_{i,j=1}^d \! \frac{\partial^2}{\partial x_i\partial x_j}\Big(a_{ij} \rho_t(x)\Big) \!-\!\sum_{i=1}^d\! \frac{\partial}{\partial x_i}\big(b_i(t,x,\rho_t(x)\d x,\alpha_t) \rho_t(x)\big).
\end{split}
\end{equation}

\begin{mylem}\label{lem-7}
Suppose the conditions $(\mathbf{A}_1)$ and $(\mathbf{A}_2)$ hold. Then, for every $\mu_0\in\pb_1^r(\R^d)$ and every $\Theta\in \Pi_{s,\mu_0}$, the law $\law_{X_t}$ of the controlled process $X_t$ as a solution to  \eqref{a-1} with initial value $\law_{X_s}=\mu_0$ belongs to $\pb_1^r(\R^d)$ for every $t\in [s,T]$.
\end{mylem}

\begin{proof}
  Let $(X_t)_{t\in [s,T]}$ be a controlled process corresponding to a feedback control $\Theta\in \Pi_{s,\mu_0}$ in the form $\alpha_t=F_t(X_t)$, that is,
  \[\d X_t=b(t,X_t,\mu_t,F_t(X_t))\d t +\sigma \d W_t,\]
  where $\mu_t=\law_{X_t}$. It is standard to show that  $\E|X_t|<\infty$ due to $(\mathbf{A}_1)$, which implies that   $\law_{X_t}$  belongs to $\pb_1(\R^d)$. $(\mathbf{A}_2)$ means that $\law_{X_t}$ admits a density $\rho_t\in \mathcal{C}^1(\R^d)$.  Using the idea of decoupling method, $(X_t)$ is also the solution to the following time-inhomogeneous SDE
  \[\d \wt{X}_t=b(t,\wt{X}_t,\mu_t,F_t(\wt{X}_t))\d t+\sigma \d W_t,\quad \wt X_s=X_s,\]
  for which we can apply the known results on heat kernel estimates for time-inhomogeneous diffusion processes.
  Let $p(s,x;t,y)$ denote the transition probability of the Markov process $(\wt X_t)$. According to \cite[Theorem A]{Zh97}, under conditions  $(\mathbf{A}_1)$ and $(\mathbf{A}_2)$, there exist constants $\kappa_1,\kappa_2>0$, independent of the choice of $\Theta $, such that
\begin{gather}\label{heat-1}
  \frac{1}{\kappa_1(t-s)^{d/2}}\exp\Big( \!-\frac{|x-y|^2}{\kappa_2(t-s)} \Big)\leq p(s,x;t,y)\leq \frac{\kappa_1}{(t-s)^{d/2}} \exp\Big(\! -\kappa_2 \frac{|x-y|^2}{t-s}\Big),\\ \label{heat-2}
  |D_yp(s,x;t,y)|\leq \frac{\kappa_1}{(t-s)^{(d+1)/2}}\exp\Big( -\kappa_2\frac{|x-y|^2}{t-s}\Big)
\end{gather} for all $x,y\in \R^d$, $0<t-s\leq T$,
  where the condition that the drift satisfies condition K  in \cite{Zh97}  is ensured by the boundedness of $b$ due to $(\mathbf{A}_1)$.
  Then, the law of $X_t$ coincides with the law of $\wt X_t$ and can be expressed by
  \begin{equation}\label{heat-3}
  \rho_t(y)=\int_{\R^d}\!p(s,x;t,y)\rho_0(x)\d x.
  \end{equation}
  As $\rho_0\in \mathscr{D}_1^r$, applying \eqref{heat-1}, \eqref{heat-3} and H\"older's inequality,
  \begin{align*}
    \int_{\R^d}\! \e^{|y|} \rho_t(y)^2 \d y&\leq \frac{\kappa_1^2}{(t-s)^{d }}\Big(
    \int_{\R^d}\e^{-\kappa_2\frac{|x-y|^2}{t-s}}\d x\Big)  \int_{\R^d}\int_{\R^d}\e^{|y|}  \e^{-\kappa_2\frac{|x-y|^2}{t-s}}\rho_0(x)^2\d y \d x\\
    &=\frac{\kappa_1^2C_1}{(t-s)^{d }}\int_{\R^d}\int_{\R^d}\e^{|x+z|} \e^{-\kappa_2\frac{|z|^2}{t-s}}\rho_0(x)^2\d z \d x\\
    &\leq \frac{\kappa_1^2C_1}{(t-s)^{d }}\int_{\R^d}\e^{|z|- \kappa_2\frac{|z|^2}{t-s}} \d z\int_{\R^d} \e^{|x|}\rho_0(x)^2\d x <\infty,
  \end{align*} where $C_1=\int_{\R^d}\e^{-\kappa_2\frac{|x-y|^2}{t-s}}\d x <\infty$.
  Similarly, we can check that $\int_{\R^d}\e^{|y|} |D\rho_t(y)|^2\d y<\infty$.
  This means that for every $t\in (s,T]$, $\law_{X_t}$ lies in $\pb_1^r(\R^d)$ as desired.
\end{proof}

\begin{mydefn}\label{def-7}
A function $\psi:[0,T]\!\times \!\mathscr{D}_1^r \to \R$ is said to be in $C^{1,1}([0,T]\!\times\!\mathscr{D}_1^r )$ if \begin{itemize}
\item[$(i)$] $\psi(\cdot,\rho)$ is continuously differentiable for every $\rho\in\!\mathscr{D}_1^r$;
\item[$(ii)$] for each $t\in [0, T]$, $\psi(t,\cdot)$ has Mortensen's derivative $\big(\frac{\delta \psi(t,\rho)}{\delta \rho(x)},\frac{\delta \psi(t,\rho)}{\delta \rho'(x)}\big)$ at every point $\rho\in \!\mathscr{D}_1^r$, which satisfies
    that for every $t\in [0,T]$, $i=1,\ldots, d$,  and $\rho\in \mathscr{D}_1^r$,
\begin{gather*}
\int_{\R^d}\! \Big|\frac{\partial}{\partial x_i}\Big(\frac{\delta \psi(t,\rho)}{\delta \rho(x)}\Big)\Big|^2\gamma(x)^{-1}\d x<\infty,\\ \lim_{\|\tilde \rho-  \rho\|_{H_1^2(\gamma)}\to 0}\int_{\R^d}\! \Big| \frac{\partial }{\partial x_i }\Big(\frac{\delta \psi(t,\rho)}{\delta \rho(x)}\Big)-\frac{\partial }{\partial x_i }\Big(\frac{\delta \psi(t,\tilde \rho)}{\delta \rho(x)}\Big)\Big|^2 \gamma(x)^{-1}\d x=0.
\end{gather*}
\end{itemize}
\end{mydefn}

Consider the HJB equation
\begin{equation}\label{f-7}
\left\{\!\begin{array}{l}
   -\partial_t u(t,\rho)-\!\inf\limits_{\alpha\in \pb(U)} \!\! \mathcal{H}\big(t,\rho, \frac{\delta u(t,\rho)}{\delta \rho(x)},\alpha\big) =0,    \quad t\in\! [0,T),\rho  \in\! \mathscr{D}_1^r, \\
   u(T, \rho)=\int_{\R^d} g(x,\rho)\rho(x)\d x,
\end{array}\right.
\end{equation} where
\begin{equation}\label{f-7.5}
\begin{split}
 \mathcal{H}\Big(t,\rho, \frac{\delta u(t,\rho)}{\delta \rho(x)},\alpha\Big) =& \int_{\R^d}\sum\limits_{i=1}^d b_i(t,x,\rho,\alpha)\frac{\partial}{\partial x_i}\Big(\frac{\delta u(t,\rho)}{\delta \rho(x)}\Big) \rho(x)\d x\\
   & -\frac 12\int_{\R^d}\! \sum\limits_{i,j=1}^d\!a_{ij}\frac{\partial }{\partial x_j }\Big(\frac{\delta u(t,\rho)}{\delta \rho(x)}\Big)  \frac{\partial \rho(x) }{\partial x_i}\d x \!+\!\! \int_{\R^d}\! f(t, x, \rho,\alpha) \rho(x)\d x.
\end{split}
\end{equation}
Now, let us introduce the definition of viscosity solution  to the HJB equation \eqref{f-7}.

\begin{mydefn}\label{def-5}
Let $u:[0,T]\times \! \mathscr{D}_1^r \to \R$ be a continuous function.
\begin{itemize}
  \item[$(i)$] $u$ is called a viscosity subsolution to \eqref{f-7} if $u(T,\rho)=\int_{\R^d}g(x,\rho)\rho(x)\d x$, and
  \begin{align}\label{e-2}
    &-\!\partial_t \psi(t_0, \rho_{0} )\!-\!\inf_{\alpha\in \pb(U)}
    \mathcal{H}\Big(t_0,\rho_{0}, \frac{\delta \psi(t_0,\rho_{0} )}{\delta \rho  (x)},\alpha\Big)
    \leq 0
  \end{align}
 for all $\psi\in\! C^{1,1}([0,T)\!\times\! \mathscr{D}_1^r)$ and all $(t_0,\rho_0) $ being a maximum point of $u-\psi$ over $[0,T)\times\! \mathscr{D}_1^r$.

  \item[$(ii)$] $u$ is called a viscosity supersolution to \eqref{f-7} if $u(T,\rho)=\int_{\R^d}g(x,\rho)\rho(x)\d x$, and
      \begin{align}\label{e-3}
    &-\partial_t \psi(t_0, \rho_{0} )-\inf_{\alpha\in \pb(U)} \mathcal{H}\Big(t_0,\rho_{0}, \frac{\delta \psi(t_0,\rho_{0} )}{\delta \rho  (x)},\alpha\Big)\geq 0
  \end{align}
  for all $\psi\in\! C^{1,1}([0,T)\!\times\! \mathscr{D}_1^r )$ and all $(t_0,\rho_0)$ being a minimum point of $u-\psi$ over $[0,T)\times\! \mathscr{D}_1^r$.
  \item[(iii)] If $u$ is both a viscosity subsolution and a viscosity supersolution to \eqref{f-7}, then $u$ is called a viscosity solution to \eqref{f-7}.
\end{itemize}
\end{mydefn}

\begin{mylem}\label{lem-6}
Assume $(\mathbf{A}_1)$-$(\mathbf{A}_3)$  hold, then for any initial distribution  $\mu_{t_0}(\d x)\!=\!\rho_{t_0}(x)\d x\!\in \! \pb_1^r(\R^d)$, $t_0\in [0,T)$,  any $\Theta\in \Pi_{t_0,\mu_{t_0}}$, the associated controlled stochastic process $(X_t)$ satisfies
\begin{gather}\label{f-6}
\lim_{s\to t} \|\rho_s -\rho_t \|_{L^2(\gamma)}+\|D \rho_s-D \rho_t\|_{L^2(\gamma)}=0,
\end{gather} and the convergence is uniform  w.r.t. $\Theta\in \Pi_{t_0,\mu_{t_0}}$,
where  $\rho_t$ stands for the   density of $\law_{X_t}$.
\end{mylem}

\begin{proof}
  Recall the representation of $\rho_t$ based on the application of decoupling method given by \eqref{heat-3}, that is,
  \[\rho_t(y)=\int_{\R^d}\! p(t_0,x; t,y)\rho_{t_0}(x)\d x.\]
Under $(\mathbf{A}_1)$, $(\mathbf{A}_2)$, $p(t_0,x;t,y)$ admits Gaussian bound estimates (cf. \cite[Chapter 9, Theorem 8]{Fri}): there exist constants $C_1,C_2>0$ such that for nonnegative integers $m_0$, $m=(m_1,\ldots,m_k)$,
\begin{equation}\label{ff-1}
\big|\partial_t^{m_0}\partial_y^m p(t_0,x; t, y)\big|\leq C_1(t-t_0)^{-(d+2m_0+|m|)/2} \exp\Big\{-C_2\frac{|y-x|^2}{t-t_0}\Big\}.
\end{equation}
Note that $C_1,C_2$ do not depend on the choice of $\Theta\!\in \!\Pi_{t_0,\mu_{t_0}}$.
Based on this estimate, direct calculation yields
\begin{align*}
  &\int_{\R^d}|\rho_t(y)-\rho_s(y)|^2\gamma(y)\d y\\
  &\leq \int_{\R^d}\!\int_{\R^d}\! (p(t_0,x;t,y)-p(t_0,x;s,y))^2\rho_{t_0}(x)\gamma(y)\d x\d y\\
  &\leq \int_{\R^d}\!\int_{\R^d}\!\Big(\int_s^t\partial_r p(t_0,x;r,y)\d r\Big)^2\rho_{t_0}(x)\gamma(y)\d x\d y\\
  &\leq C_1^2(t\!-\!s)\int_{\R^d}\!\int_{\R^d}\!\int_s^t \frac{1}{(r-t_0)^{ d +1}}\e^{-2C_2 \frac{|y-x|^2}{r-t_0}} \rho_{t_0}(x)\gamma(y) \d r \d x \d y\\
  &\leq C_1^2(t\!-\!s)\int_{\R^d}\!\int_{\R^d}\!\int_s^t \frac 1{(r-t_0)^{\frac d2+1}} \e^{-2C_2|z|^2} \rho_{t_0}(x) \gamma (x\!+\!\sqrt{r\!-\!t_0} z)\d r \d x \d z\\
  &\leq C(t\!-\!s) \big((s\!-\!t_0)^{-\frac d2}\!-\!(t\!-\!t_0)^{-\frac d2}\big)\int_{\R^d}\! \rho_{t_0}(x)\gamma(x)\d x \int_{\R^d}\!\gamma\big(\sqrt{T\!-\!t_0}|z|\big) \e^{-2C_2|z|^2}\d z\\
  &\leq C(t\!-\!s) \big((s\!-\!t_0)^{-\frac d2}\!-\!(t\!-\!t_0)^{-\frac d2}\big).
\end{align*}
Analogously, using H\"older's inequality,
\begin{align*}
  &\int_{\R^d}\!|D\rho_t(y)-D\rho_s(y)|^2\gamma(y)\d y\\
  &\leq (t\!-\!s) \int_{\R^d}\!\int_{\R^d}\!\int_s^t\!|\partial_r\partial_y p(t_0,x;r,y)|^2 \rho_{t_0}(x)\gamma(y)\d r\d x\d y\\
  &\leq C_1^2(t\!-\!s)\int_{\R^d}\!\int_{\R^d}\! \frac{1}{(r-t_0)^{d+3}}\e^{-2C_2\frac{|y-x|^2}{r-t_0}} \rho_{t_0}(x)\gamma(y)\d r\d x\d y\\
  &\leq C(t\!-\!s)\big((s\!-\!t_0)^{-\frac{d}2 -2}\!-\!(t\!-\!t_0)^{-\frac  d2 -2}\big) \int_{\R^d}\!\rho_{t_0}(x)\gamma(x)\d x\int_{\R^d}\! \gamma(\sqrt{T\!-\!t_0}|z|)\e^{-2C_2|z|^2}\d z\\
  &\leq  C(t\!-\!s)\big((s\!-\!t_0)^{-\frac d2 -2}\!-\!(t\!-\!t_0)^{-\frac d2 -2}\big).
\end{align*}
Therefore, we obtain the conclusion \eqref{f-6}.
\end{proof}

\begin{mythm}\label{thm-3}
Suppose $(\mathbf{A}_1)$-$(\mathbf{A}_3)$  hold, then the value function $V(s,\rho)$ on $[0,T]\times \!\mathscr{D}_1^r$ of the optimal control problem \eqref{a-4} is a viscosity solution to HJB equation \eqref{f-7} in the sense of Definition \ref{def-5} and satisfies the continuity property \eqref{cont}.
\end{mythm}

\begin{proof}
The assertion that  $V$ satisfies \eqref{cont} can be proved by checking directly the conditions of Proposition \ref{lem-2}.

  \emph{Viscosity subsolution} Let $(t_0,\rho_{t_0})\in [0,T)\times\!\mathscr{D}_1^r $  and $\psi\in C^{1,1}([0,T)\times \!\mathscr{D}_1^r)$ be a test function such that
  \begin{equation}\label{f-8}
  0=(V-\psi)(t_0,\rho_{t_0})=\max\big\{ (V-\psi)(t,\rho); \ (t, \rho)\in [0,T)\times \!\mathscr{D}_1^r\big\}.
  \end{equation}
  Due to Proposition \ref{lem-3},  using the control $\alpha_t\equiv \alpha\in \pb(U)$, we derive that
  \[V(t_0,\rho_{t_0} )\leq \E_{t_0,\mu_{t_0}}\Big[\int_{t_0}^t\!\! f(r,X_r ,\rho_r ,\alpha)\d r+V(t, \rho_t )\Big],\]
  where $\mu_{t_0}(\d x)=\rho_{t_0}(x)\d x$, and $\rho_t$ stands for the density of $\law_{X_t}$. Due to \eqref{f-8}, we get
  \begin{equation}\label{f-9}
  \psi(t,\rho_t )-\psi(t_0,\rho_{t_0} )
  +  \int_{t_0}^t\!\int_{\R^d}\! f(r,x,\rho_r , \alpha)\rho_r (x)\d x \d r \geq 0.
  \end{equation}
  Dividing both sides of \eqref{f-9} by $t-t_0$ and letting $t\downarrow t_0$, using the integration by parts formula, we derive from \eqref{f-5} that
  \begin{align*}
  &-\partial_t \psi(t_0, \rho_{t_0} )-\Big[\int_{\R^d}\!\sum_{i=1}^d b_i(t_0,x,\rho_{t_0},\alpha )\frac{\partial}{\partial x_i}\Big( \frac{\delta \psi(t_0,\rho_{t_0} )}{\delta \rho  (x)}\Big)  \rho_{t_0} (x)\d x \\
  &-\frac 12\int_{\R^d}\!\sum_{i,j=1}^d a_{ij} \frac{\partial }{\partial x_j }\Big(\frac{\delta\psi (t_0,\rho_{t_0})}{\delta\rho(x)}\Big) \frac{\partial \rho_{t_0}(x)}{\partial{x_i}}\d x +\int_{\R^d}\!f(t_0, x, \rho_{t_0} , \alpha)\rho_{t_0}(x)\d x\Big]\leq 0.
  \end{align*}
  Taking the infimum for $\alpha$ over $\pb(U)$ we obtain \eqref{e-2} and  conclude that $V$ is a viscosity subsolution of \eqref{f-7}.

  \emph{Viscosity supersolution} Let $(t_0,\rho_{t_0})\in [0,T)\times\!\mathscr{D}_1^r$ and $\psi\in C^{1,1}([0,T]\times\!\mathscr{D}_1^r)$ such that
  \begin{equation}\label{f-10}
  0=(V-\psi)(t_0,\rho_{t_0})=\min\big\{(V-\psi)(t,\rho);\ (t,\rho )\in [0,T)\times\!\mathscr{D}_1^r\big\}.
  \end{equation}

  Let us first show that the function
  \[\rho\mapsto \inf_{\alpha\in \pb(U)}\!\mathcal{H}\Big(t,\rho,\frac{\delta \psi(t,\rho)}{\delta \rho(x)},\alpha\Big)\]
  is continuous in  $ \mathscr{D}_1^r $ when $ \mathscr{D}_1^r$ is endowed with the weighted Sobolev norm $\|\,\cdot\,\|_{H_1^2(\gamma)}$.

  In fact, by $(\mathbf{A}_3)$, for any    $t\in [0,T]$, $\alpha\in \pb(U)$
  \begin{equation}\label{f-11}
  \begin{aligned}
    &\Big|\int_{\R^d}\! f(t,x,\rho,\alpha)\rho(x)\d x-\int_{\R^d}\! f(t,x,\tilde \rho,\alpha)\tilde \rho(x)\d x\Big|\\
    &\leq  \int_{\R^d}\!|f(t,x,\rho,\alpha)- f(t,x,\tilde \rho,\alpha)|\rho(x)\d x+\int_{\R^d}\!|f(t,x,\tilde \rho, \alpha)||\rho(x)-\tilde \rho(x)|\d x\\
    &\leq \tilde K_3\|\rho-\tilde\rho\|_{L^2(\gamma)} +\tilde K_3\big(\int_{\R^d}\gamma(x)^{-1}\d x\big)\|\rho-\tilde \rho\|_{L^2(\gamma)} .
  \end{aligned}
  \end{equation}

 For  $\psi\in C^{1,1}([0,T] \times\!\mathscr{D}_1^r)$,
  since
  \begin{align*}
    &\Big|\int_{\R^d}\! \frac{\partial}{\partial x_j}\Big(\frac{\delta \psi(t,\rho)}{\delta \rho(x)}\Big)\frac{\partial}{\partial x_i}\rho(x)\d x- \int_{\R^d}\! \frac{\partial}{\partial x_j}\Big(\frac{\delta \psi(t,\tilde \rho)}{\delta \rho(x)}\Big)\frac{\partial}{\partial x_i}{\tilde \rho}(x)\d x\Big|\\
    &\leq \int_{\R^d}\!\!\Big|\frac{\partial}{\partial x_j}\Big(\frac{\delta \psi(t,\rho)}{\delta \rho(x)}\Big)- \frac{\partial}{\partial x_j}\Big(\frac{\delta \psi(t,\tilde \rho)}{\delta \rho(x)}\Big)\Big|\frac{\partial}{\partial x_i}\rho(x) \d x\\
     &\quad + \int_{\R^d}\Big|\frac{\partial}{\partial x_j}\Big(\frac{\delta \psi(t,\tilde \rho)}{\delta \rho(x)}\Big)\Big| \big|\frac{\partial}{\partial x_i}\rho(x)-\frac{\partial}{\partial x_i} \tilde\rho(x)\big|\d x\\
    &\leq  \Big(\int_{\R^d}\!|D\rho(x)|^2\gamma(x)\d x\Big)^{\frac 12} \Big(\int_{\R^d} \Big|\frac{\partial}{\partial x_j}\Big(\frac{\delta \psi(t,\rho)}{\delta \rho(x)}\Big)\!-\!\frac{\partial }{\partial x_j}\Big( \frac{\delta \psi(t,\tilde \rho)}{\delta \rho(x)}\Big)\Big|^2\gamma(x)^{-1} \d x\Big)^{\frac 12}\\
     &\quad \  +\! \Big(\int_{\R^d}\!\Big|\frac{\partial}{\partial x_j}\Big(\frac{\delta \psi(t,\tilde \rho)}{\delta \rho(x)}\Big)\Big|^2\gamma(x)^{-1}\d x\Big)^{\frac 12} \Big(\int_{\R^d} \big|\frac{\partial }{\partial x_i}\rho(x)\!-\!\frac{\partial}{\partial x_i} \tilde \rho(x) \big|^2\gamma(x)\d x\Big)^{\frac 12},
  \end{align*}
  we have, if $ \|\tilde \rho - \rho\|_{H_1^2(\gamma)}\to 0$, then
  \[\int_{\R^d}\sum_{i,j=1}^d a_{ij}\frac{\partial}{\partial x_j}\Big( \frac{\delta \psi(t,\rho)}{\delta \rho(x)}\Big) \frac{\partial}{\partial x_i}\rho(x)\d x - \int_{\R^d}\sum_{i,j=1}^d a_{ij}\frac{\partial}{\partial x_j}\Big( \frac{\delta \psi(t,\tilde \rho)}{\delta \rho(x)}\Big) \frac{\partial}{\partial x_i}\tilde\rho(x)\d x \longrightarrow  0.
  \]
  Moreover, by $(\mathbf{A}_1)$,
  \begin{align*}
    &\Big|\int_{\R^d}\! b_i(t,x,\rho,\alpha)\frac{\partial}{\partial x_i}\Big(\frac{\delta \psi(t,\rho)}{\delta \rho(x)}\Big) \rho(x)\d x -\int_{\R^d}\!b_i(t,x,\tilde \rho,\alpha)\frac{\partial}{\partial x_i}\Big(\frac{\delta\psi(t,\tilde \rho)}{\delta \rho(x)}\Big) \tilde \rho(x)\d x\Big|\\
    &\leq \Big| \int_{\R^d}\!\big(b_i(t,x,\rho,\alpha)-b_i(t,x,\tilde \rho,\alpha)\big)\frac{\partial}{\partial x_i}\Big( \frac{\delta\psi(t,\rho)}{\delta \rho(x)}\Big) \rho(x)\d x\Big|\\
    &\quad +\Big|\int_{\R^d}\! b_i(t,x,\tilde \rho,x)\Big( \frac{\partial}{\partial x_i}\Big( \frac{\delta\psi(t,\rho)}{\delta \rho(x)}\Big)-\frac{\partial}{\partial x_i}\Big( \frac{\delta\psi(t,\tilde \rho)}{\delta \rho(x)}\Big)\Big) \rho(x)\d x\Big|\\
    &\quad + \Big| \int_{\R^d}\! b_i(t,x,\tilde \rho,\alpha)\frac{\partial}{\partial x_i}\Big( \frac{\delta\psi(t,\tilde \rho)}{\delta \rho(x)}\Big) (\rho(x)-\tilde \rho(x)) \d x \Big|\\
    &\leq \tilde {K}_1     \|\rho-\tilde \rho\|_{L^2(\gamma)} \Big(\int_{\R^d}\!\Big|\frac{\partial}{\partial x_i}\Big(\frac{\delta \psi(t,  \rho)}{\delta \rho(x)}\Big)\Big|^2\gamma(x)^{-1}\d x\Big)^{\frac 12}\|\rho\|_{L^2(\gamma)}\\
    &\quad + \tilde K_2 \|\rho\|_{L^2(\gamma)} \Big(\int_{\R^d} \Big|\frac{\partial}{\partial x_j}\Big(\frac{\delta \psi(t,\rho)}{\delta \rho(x)}\Big)\!-\!\frac{\partial }{\partial x_j}\Big( \frac{\delta \psi(t,\tilde \rho)}{\delta \rho(x)}\Big)\Big|^2 \gamma(x)^{-1} \d x\Big)^{\frac 12}\\
    &\quad +\tilde K_2  \Big(\int_{\R^d}\!\Big|\frac{\partial}{\partial x_j}\Big(\frac{\delta \psi(t,\tilde \rho)}{\delta \rho(x)}\Big)\Big|^2\gamma(x)^{-1}\d x\Big)^{\frac 12}   \|\rho-\tilde \rho\|_{L^2(\gamma)}.
  \end{align*}
  This shows the uniform continuity in  $\alpha\in \pb(U)$ of the term
  \begin{align*}
    \rho \mapsto \int_{\R^d}\sum_{i=1}^db_i(t,x,\rho,\alpha) \frac{\partial}{\partial {x_i}} \Big(\frac{\delta\psi(t,\rho)}{\delta \rho(x)}\Big)\rho(x)\d x.
  \end{align*}
  As a consequence,
  invoking the uniform continuity  of these three terms in $\mathcal{H}\big(t,\rho,\frac{\delta\psi(t,\rho)}{\delta \rho(x)}, \alpha\big)$, we can show the continuity of $  \inf_{\alpha\in \pb(U)}\mathcal{H}\big(t,\rho,\frac{\delta\psi(t,\rho)}{\delta \rho(x)}, \alpha\big)$.

  Next, we shall prove
  \begin{equation}\label{f-12}
  -\partial_t\psi(t_0,\rho_{t_0})-\inf_{\alpha\in \pb(U)}\!\mathcal{H}\Big(t_0,\rho_{t_0},\frac{\delta \psi(t_0,\rho_{t_0})}{\delta \rho(x)},\alpha\Big)\geq  0
  \end{equation} by contradiction.
  Suppose that
  \begin{equation}\label{f-13}
  -\partial_t\psi(t_0,\rho_{t_0})-\inf_{\alpha\in \pb(U)}\!\mathcal{H}\Big(t_0,\rho_{t_0},\frac{\delta \psi(t_0,\rho_{t_0})}{\delta \rho(x)},\alpha\Big)<0.
  \end{equation}
  Then, using the continuity of $(t,\rho)\mapsto \partial_t\psi(t,\rho)+\inf_{\alpha\in\pb(U)} \mathcal{H}\big(t,\rho,\frac{\delta \psi(t,\rho)}{\delta\rho(x)},\alpha\big)$ at $(t_0,\rho_{t_0})$, there exist $\veps,\tilde \zeta_1>0$ such that for any $(t, \rho )\in [0,T]\!\times\!\mathscr{D}_1^r$ satisfying $|t-t_0|<\tilde \zeta_1$, $\|\rho-\rho_{t_0}\|_{H_1^2(\gamma)} <\tilde \zeta_1$, it holds
  \begin{equation}\label{f-14}
  -\partial_t\psi(t ,\rho )-\inf_{\alpha\in \pb(U)}\!\mathcal{H}\Big(t ,\rho ,\frac{\delta \psi(t ,\rho )}{\delta \rho(x)},\alpha\Big)\leq -\veps.
  \end{equation}

  For   $\Theta\in \!\Pi_{t_0,\mu_{t_0}}$, denote   by $(X_t )$ the associated controlled stochastic process. By  Lemma \ref{lem-6},   there exists a constant $\tilde \zeta_2>0$, independent the choice $\Theta\in \! \Pi_{t_0,\mu_{t_0}}$, such that for $|t-t_0|\leq \tilde \zeta_2$,
  \begin{equation}\label{f-14.5}
  \|\rho_t- \rho_{t_0}\|_{L^2(\gamma)}+\|D\rho_t- D\rho_{t_0}\|_{L^2(\gamma)}<\tilde \zeta_1.
  \end{equation}

Take two sequence $\delta_n,\gamma_n, n\geq 1$ satisfying $0<\delta_n<\tilde \zeta_1\wedge \tilde \zeta_2$, $\gamma_n>0$ and
\begin{equation}\label{f-15}
   \lim_{n\to \infty} \gamma_n/\delta_n=0.
\end{equation}
By the dynamic programming principle, there exists a sequence of Markovian feedback controls $\Theta_n\in \Pi_{t_0,\mu_0}$ such that
\begin{equation*}
 V(t_0,\rho_{t_0})\geq \E_{t_0,\mu_{t_0}}\Big[\int_{t_0}^{t_0+\delta_n} \!\!f(r,X_r^n,\rho_r^n,\alpha_r^n)\d r+\!V(t_0+\delta_n, \rho_{t_0+\delta_n}^n)\Big]-\gamma_n,
\end{equation*} where $(X_r^n,\alpha_r^n)$ is the stochastic processes associated with the control $\Theta_n$, and $\rho_r^n$ denotes the density of $\law_{X_r^n}$. By virtue of \eqref{f-10},
   \begin{equation}\label{f-16}
   \psi(t_0,\rho_{t_0})\geq \E_{t_0,\mu_{t_0}}\Big[\int_{t_0}^{t_0+\delta_n}\!\! f(r, X_r^n,\rho_r^n,\alpha_r^n)\d r+ \psi(t_0+\delta_n, \rho_{t_0+\delta_n}^n)\Big]-\gamma_n.
   \end{equation}
Since $\psi\in C^{1,1}([0,T]\!\times\!\mathscr{D}_1^r) $ and $\rho_t^n$ satisfies the nonlinear Fokker-Planck equation \eqref{f-5}, we have
\begin{align*}
  \frac{\gamma_n}{\delta_n}&\geq \frac1{\delta_n}\E_{t_0,\mu_0}\Big[\int_{t_0}^{t_0+\delta_n}\!\! f(r,X_r^n,\rho_r^n,\alpha_r^n)\d r+\int_{t_0}^{t_0 +\delta_n} \!\!\frac{\d}{\d r}\big( \psi(r, \rho_r^n)\big)\d r\Big]\\
  &=\frac{1}{\delta_n}\!\int_{t_0}^{t_0+\delta_n} \!\!\! \int_{\R^d}\! \!f(r,x,\rho_r^n,\alpha_r^n) \rho_r^n(x)\d x \d r \\ &\qquad +\!\frac{1}{\delta_n}\!\int_{t_0}^{t_0+\delta_n}\!\! \Big[\partial_r\psi(r,\rho_r^n)\!+\!\!\int_{\R^d}\! \!  \Big(\frac{\delta \psi(r,\rho_r^n)}{\delta\rho(x)}\Big) \mathcal{L}_\alpha^\ast\rho_r^n(x)\d x\Big]\d r\\
  &\geq \frac{1}{\delta_n} \int_{t_0}^{t_0+\delta_n}\!\!\Big[\partial_r\psi(r,\rho_r^n) +\inf_{\alpha\in \pb(U)}\!\mathcal{H}\Big(r,\rho_r^n,\frac{\delta \psi(r,\rho_r^n)}{\delta \rho(x)},\alpha\Big)\Big]\d r\\
  &\geq \frac1{\delta_n}\int_{t_0}^{t_0+\delta_n}\veps \d r \qquad \qquad \qquad \text{(due to  \eqref{f-14}, \eqref{f-14.5})}\\
  &=\veps>0.
\end{align*} Letting $n\to \infty$, this contradicts \eqref{f-15}, and hence the assertion \eqref{f-13} is false. Therefore, $V(t,\mu)$ is a viscosity supersolution to \eqref{f-7}. Invoking the discussion in the first part, $V$ is  further a viscosity solution to \eqref{f-7}.
\end{proof}

\section{Comparison principle for HJB equations}

The space $\pb_1^r(\R^d)$ is an infinite dimensional space, and the bounded sets in $\pb_1^r(\R^d)$ are not necessary precompact. Compared with finite dimensional spaces, this essential  difference make the study of viscosity solution theory for HJB equations on infinite dimensional space more difficult.  To overcome this difficulty, we shall use Borwein-Preiss variational principle (cf. \cite{AF14,BP87}). Another crucial point for us to establish the comparison principle on the Wasserstein space is to find suitable smooth functions, which  on the one hand admit  Mortensen's derivatives, on the other hand can distinguish two different probability measures in $\pb_1^r(\R^d)$. The simplicity of Mortensen's derivative facilitates our construction of necessary smooth functions, which is our initial purpose to introduce Mortensen's derivative.

The $L^1$-Wasserstein distance on $ \pb_1^r(\R^d)$  can induce a distance on  $\mathscr{D}_1^r$ via
\[\mathrm{dist}(\rho,\tilde \rho):=\W_1(\rho(x)\d x,\tilde \rho(x)\d x),\quad \rho,\,\tilde \rho\in \mathscr{D}_1^r.\]
Since HJB equation \eqref{f-7} depends on $D\rho$, the completeness of  $ \pb_1^r(\R^d)$ or $\mathscr{D}_1^r$ relative to $\W_1$ or $\mathrm{dist}(\cdot,\cdot)$ cannot ensure the existence of $D\rho$ for the limit point, and so it is inappropriate to use $\mathrm{dist}(\cdot,\cdot)$.  In this work we consider the completeness of $\mathscr{D}_1^r$ w.r.t.\,the weighted Sobolev norm $\|\,\cdot\,\|_{H_1^2(\gamma)}$, which is denoted by
$\bar{\mathscr{D}}_1^r$ in the sequel.
For $\rho\in \bard$, $\rho$ is in the Sobolev space $H_1^2(\R^d)$ and is not necessary in $\mathcal{C}^1(\R^d)$, but it still holds that
\[\rho\geq 0, \ \int_{\R^d} \rho(x)\d x=1, \ \int_{\R^d}\big(\rho(x)^2+|D\rho(x)|^2\big)\e^{|x|}\d x<\infty.\]
Indeed, if $\rho_n\in\mathscr{D}_1^r$ converges to $\rho$ in the norm $\|\,\cdot\,\|_{H_1^2(\gamma)}$,
\begin{gather*}
\int_{\R^d}\!|\rho_n(x)-\rho(x)|\d x\leq \Big(\int_{\R^d}\!\gamma(x)^{-1}\d x\Big)^{\frac 12} \Big(\int_{\R^d}|\rho_n(x)-\rho(x)|^2\gamma(x)\d x\Big)^{\frac 12},
\end{gather*} which means that $\rho_n$ converges to $\rho$ in $L^1(\R^d)$, and hence    $\rho\geq 0$ and $\int_{\R^d}\rho(x)\d x=1$ as $\rho_n$ are probability densities.

The conditions $(\mathbf{A}_1)$, $(\mathbf{A_3})$ on the coefficients $b,\,f,\,g$ can be continuously extended to the space $\bard$. Correspondingly, the value function and the HJB equation \eqref{f-7} can also be continuously extended to $\bard$. In the following we shall establish a comparison principle for the HJB equation \eqref{f-7} on $\bard\supset \mathscr{D}_1^r$.


\begin{mylem}\label{lem-8}
Define a functional $F$ on $\bar{\mathscr{D}}_1^r$ by
\[F(\rho)=\int_{\R^d}\! \big(  \rho(x)^2+|D\rho(x)|^2\big) \gamma(x)\d x.\]
Then, for $\hat\rho\in \bard$,
\begin{equation}\label{h-0}
\frac{\delta F(\rho-\hat\rho)}{\delta \rho(x)}= 2(\rho(x)-\hat\rho(x))\gamma(x),\quad \frac{\delta F(\rho-\hat\rho)}{\delta \rho'(x)}=2 (D\rho(x)-D\hat{\rho}(x))\gamma(x),
\end{equation}
and $F\in  {C}^{1,1}([0,T]\times \bard)$.
\end{mylem}

\begin{proof}
  For $\phi_\veps\in H_1^2(\R^d)$ with $\|\phi_\veps\|_{H_1^2(\gamma)}\to 0$ as $\veps\to 0$,
  \begin{align*}
    &F(\rho\!-\!\hat\rho\! +\!\phi_\veps)\!-\! F(\rho\!-\!\hat \rho)\!-\!\int_{\R^d}\! \big[2(\rho(x)\!- \!\hat{\rho}(x))\phi_\veps (x)\! -\! 2\la D\rho(x)\!-\! D\hat{\rho}(x), D\phi_\veps(x)\raa \big]\gamma(x)\d x\\
    &=\int_{\R^d} \big[   \phi_\veps(x)^2   \!+\!|D\phi_\veps(x)|^2\big]\gamma(x)\d x=\|\phi_\veps\|_{H_1^2(\gamma)}^2,
  \end{align*}
  which implies \eqref{h-0} immediately.

  Now we check the conditions in Definition \ref{def-7}(ii). For $\rho,\,\tilde \rho \in \bard$, by \eqref{h-0} and \eqref{norm-1},
  \begin{align*}
     &\int_{\R^d}\Big|D\Big(\frac{\delta F(\rho\!-\!\hat \rho)}{\delta\rho(x)}\Big)-D\Big(\frac{\delta F(\tilde \rho\!-\!\hat\rho)}{\delta \rho(x)}\Big)\Big|^2 \gamma(x)^{-1}\d x\\
     &=4\int_{\R^d}\big|D(\rho-\tilde \rho)(x)\gamma(x)+(\rho-\tilde \rho)(x)D\gamma(x)\big|^2\gamma(x)^{-1}\d x\\
     &\leq 8\int_{\R^d} |D\rho(x)-D\tilde \rho(x)|^2\gamma(x)\d x+ 8\kappa^2\int_{\R^d}|\rho(x)-\tilde \rho(x)|^2\gamma(x)\d x\\
     &\leq C\|\rho-\tilde \rho\|_{H_1^2(\gamma)},
  \end{align*} which tends to $0$ as $\|\rho-\tilde \rho\|_{H_1^2(\gamma)}\to 0$.
  In the same way,
   \begin{align*}
     \int_{\R^d}\Big|D\Big(\frac{\delta F(\rho)}{\delta\rho(x)}\Big)\Big|^2\gamma(x)^{-1}\d x
     &\leq 4\int_{\R^d} |D \rho (x)\gamma(x)+ \rho (x) D\gamma(x)|^2\gamma(x)^{-1}\d x\\
     &\leq C\|\rho \|_{H_1^2(\gamma)}^2<\infty.
     \end{align*}
   Consequently, we conclude $F\in C^{1,1}([0,T]\times \bard)$.
\end{proof}

\begin{mylem}[Borwein-Preiss variational principle]\label{lem-9}
Let $(Y, d_Y)$ be a complete metric space and let $F:Y\mapsto \R\cup \{-\infty\}$ be an upper semicontinuous function, $F\not\equiv -\infty$, uniformly bounded from above.  Let $\veps>0$ and $y_0\in Y$ be such that
\[F(y_0)\geq \sup_{y\in Y} F(y)-\veps.\]
Then there exist $y_k\in Y$, $y_\veps\in Y$ and nonnegative numbers $\beta_k$ with $\sum_{k=1}^\infty \beta_k=1$ such that
\begin{gather*}
  \lim_{k\to \infty} d_Y(y_k,y_\veps)=0,\qquad \sup_{k\geq 1} d_Y(y_k,y_\veps)\leq \veps^{1/4}, \quad d_Y(y_\veps,y_0)\leq \veps^{1/4},\\
  F(y_\veps)\geq \sup_{y\in Y} F(y)-\veps, \quad F(y_\veps)-\sqrt{\veps} \Delta(y_\veps)\geq F(y)-\sqrt{\veps}\Delta(y),\ \forall\,y\in Y,
\end{gather*}  where $\Delta(y):=\sum_{k=1}^\infty \beta_k d_Y^2(y,y_k)$.
\end{mylem}
We refer the reader to  \cite[Proposition A.2]{AF14} for the proof of this lemma.

\begin{mythm}[Comparison Principle]\label{thm-4}
Assume that $(\mathbf{A}_1)$-$(\mathbf{A}_3)$   hold.  Let $W$ and $V$ be respectively a viscosity subsolution and supersolution to \eqref{f-7} both satisfying the continuity condition
\begin{equation}\label{reg}|V(t,\rho)-V(s,\tilde \rho)|\leq C(|t-s|^{\frac 12}+\|\rho-\tilde\rho\|_{L^2(\gamma)}), \ s,t\in [0,T], \rho,\tilde \rho\in \bard,
\end{equation} where $C$ is a positive constant. Then
\begin{equation}\label{h-1}
W(t,\rho)\leq V(t,\rho),\qquad t\in [0,T),\ \rho\in \bard.
\end{equation}
\end{mythm}

\begin{proof}
  We shall prove \eqref{h-1} by contradiction. Suppose there is  a $(\bar t,\bar \rho)\in [0,T)\times \bard$ such that
  \begin{equation}\label{h-2}
    W(\bar t,\bar \rho)>V(\bar t,\bar\rho).
  \end{equation} By the continuity condition \eqref{reg}, we can always choose $(\bar t,\bar \rho)$ with $\bar t>0$.

  Let $S=[0,T]^2\times  \bard\times \bard$ and define a distance on $S$ by
  \[d_{S}((s,t,\rho,\chi),(\tilde s,\tilde t,\tilde \rho,\tilde \chi)):=\big(|s-\tilde s|^2\!+\!|t-\tilde t|^2\!+\!\|\rho\!-\!\tilde \rho\|_{H_1^2(\gamma)}\!+\!\|\chi\!-\tilde \chi\|_{H_1^2(\gamma)}\big)^{\frac 12}.\]
 Consider the auxiliary function
 \begin{align*}
   \Phi(t,s,\rho,\chi)&=W(t,\rho)-V(s,\chi)-\tilde \alpha\e^{\eta(2T-t-s)}\big(\|\rho\|_{H_1^2(\gamma)}^2\!+\! \|\chi\|_{H_1^2(\gamma)}\big)\\
   &\quad -\beta (2T-s-t)-\frac{\lambda}{t}-\frac{\lambda}{s}-\frac1{2\theta}
   \big( F(\rho-\chi)+|t-s|^2\big),
 \end{align*}where parameters $\tilde \alpha,\,\beta,\,\lambda,\,\theta\in (0,1)$, $\eta>0$, and the functional $F$ is given in Lemma \ref{lem-8}. By virtue of Borwein-Preiss variational principle, Lemma \ref{lem-9}, for any $\veps\in (0,1)$, some $(\hat t,\hat s,\hat\rho,\hat\chi)\in S$ with
 \[\Phi(\hat t,\hat s,\hat \rho,\hat \chi)\geq \sup\big\{\Phi(t,s,\rho,\chi);(t,s,\rho,\chi)\in S\big\}-\veps,\]
 then there exist $(t_k,s_k,\rho_k,\chi_k)\in S$, $(t_\veps,s_\veps,\rho_\veps,\chi_\veps)\in S$, and $\beta_k\geq 0$ with $\sum_{k=1}^\infty \beta_k=1$ such that
 \begin{equation}\label{h-3}
 \begin{split}
   &\Phi(t_\veps,s_\veps, \rho_\veps,\chi_\veps)-\sqrt{\veps}\Delta(t_\veps,s_\veps,\rho_\veps, \chi_\veps)\\
   &=\sup \big\{\Phi(t,s,\rho,\chi)-\sqrt{\veps}\Delta (t,s,\rho,\chi); (t,s,\rho,\chi)\in S\big\},
 \end{split}
 \end{equation}
 where
 \[\Delta(t,s,\rho,\chi)=\sum_{k=1}^\infty \beta_k d_S((t,s,\rho, \chi),(t_k,s_k,\rho_k,\chi_k))^2.\]
 Moreover,
 \begin{equation}\label{h-3.5}
 \sup_{k\geq 1} d_S((t_k,s_k,\rho_k,\chi_k), (t_\veps,s_\veps, \rho_\veps, \chi_\veps))\leq \veps^{\frac 14}, \ \ d_S((t_\veps,s_\veps,\rho_\veps,\chi_\veps),(\hat t,\hat s,\hat \rho,\hat\chi))\leq \veps^{\frac 14}.
 \end{equation}
 Notice that the maximum point $(t_\veps,s_\veps,\rho_\veps, \chi_\veps)$ also depends on the parameters $\tilde \alpha$, $\beta$, $\lambda$,  $\theta$, and $\eta$.
 According to \eqref{h-3}, for some $\rho_0\in \bard$,
 \begin{equation}\label{h-4} \Phi(t_\veps,s_\veps,\rho_\veps,\chi_\veps)-\sqrt{\veps} \Delta(t_\veps,s_\veps,\rho_\veps,\chi_\veps)\geq \Phi(T,T,\rho_0,\rho_0)-\sqrt{\veps}\Delta(T,T, \rho_0,\rho_0).
 \end{equation}
 By the boundedness of $f$ and $g$ due to $(\mathbf{A}_3)$, there exists $M>0$ such that $\max\{|V(t,\rho)|,|W(t,\rho)|\}\leq M$ for any $(t,\rho)\in [0,T]\times \bard$.
 Direct calculation yields from \eqref{h-4} that
 \begin{align*}
   &\tilde \alpha \big(\|\rho_\veps\|_{H_1^2(\gamma)}^2+\| \chi_\veps\|_{H_1^2(\gamma)}^2\big) \\
   &\leq W(t_\veps,\rho_\veps)-V(s_\veps,\chi_\veps)+2\tilde \alpha \|\rho_0\|_{H_1^2(\gamma)}^2+\sqrt{\veps}\Delta (T,T,\rho_0,\rho_0)\\
   &\leq 2M +2\tilde \alpha\|\rho_0\|_{H_1^2(\gamma)}^2\\
   &\quad +2\sqrt{\veps}\sum_{k= 1}^\infty \beta_k\big[d_S((T,T,\rho_0,\rho_0), (t_\veps,s_\veps,\rho_\veps,\chi_\veps))^2\!+ \!d_S((t_\veps,s_\veps,\rho_\veps,\chi_\veps), (t_k,s_k, \rho_k,\chi_k))^2\big]\\
   &\leq 2M+2\tilde \alpha\|\rho_0\|_{H_1^2(\gamma)}^2+2\veps+2\sqrt{\veps} d_S((T,T,\rho_0,\rho_0), (t_\veps,s_\veps,\rho_\veps,\chi_\veps))^2\\
   &\leq 2M+(2\tilde \alpha+8\sqrt{\veps} ) \|\rho_0\|_{H_1^2(\gamma)}^2+2\veps +8\sqrt{\veps} T^2  + 4\sqrt{\veps}\big(\|\rho_\veps\|_{H_1^2(\gamma)}^2+\| \chi_\veps\|_{H_1^2(\gamma)}^2\big).
 \end{align*}
 Hence,
 \begin{equation}\label{h-5}
 (\tilde \alpha-4\sqrt{\veps}) \big(\|\rho_\veps\|_{H_1^2(\gamma)}^2+\| \chi_\veps\|_{H_1^2(\gamma)}^2\big) \leq 2M+10\|\rho_0\|_{H_1^2(\gamma)}^2 +2+8T^2=:M_1<\infty.
 \end{equation}

 Applying \eqref{h-3} again, we get
 \begin{equation}\label{h-6}
 \begin{split}
   &2\Phi(t_\veps,s_\veps,\rho_\veps,\chi_\veps)-2\sqrt{\veps} \Delta(t_\veps,s_\veps,\rho_\veps,\chi_\veps)\\
   &\geq \Phi(t_\veps,t_\veps,\rho_\veps,\rho_\veps)- \sqrt{\veps}\Delta(t_\veps,t_\veps,\rho_\veps,\rho_\veps)+ \Phi(s_\veps,s_\veps,\chi_\veps,\chi_\veps)-\sqrt{\veps} \Delta(s_\veps,s_\veps,\chi_\veps,\chi_\veps).
 \end{split}
 \end{equation}
 Invoking \eqref{reg}, this yields that
 \begin{equation}\label{h-7}
 \begin{split}
 &\frac{1}{\theta}\big( F(\rho_\veps-\chi_\veps)+|t_\veps-s_\veps|^2\big)\\
 &\leq W(t_\veps,\rho_\veps)-W(s_\veps,\chi_\veps)+ V(t_\veps,\rho_\veps)-V(s_\veps,\chi_\veps)\\
 &\quad +\sqrt{\veps}\big(
 \Delta(t_\veps,t_\veps,\rho_\veps,\rho_\veps)+\Delta(s_\veps,
 s_\veps,\chi_\veps,\chi_\veps) -2\Delta(t_\veps,s_\veps,\rho_\veps,\chi_\veps)\big)\\
 &\leq 2C\big(|t_\veps-s_\veps|^{\frac 12}\!+\!\|\rho_\veps -\chi_\veps\|_{L^2(\gamma)}\big)\\
  &\quad+\sqrt{\veps} \sum_{k=1}^\infty \beta_k\big[ d_S((t_\veps,t_\veps,\rho_\veps, \rho_\veps),(t_k,s_k,\rho_k,\chi_k))^2\!+\!
  d_S((s_\veps,s_\veps,\chi_\veps, \chi_\veps),(t_k,s_k,\rho_k,\chi_k))^2\big]\\
  &\leq 2C\big(|t_\veps-s_\veps|^{\frac 12}\!+\!\|\rho_\veps -\chi_\veps\|_{L^2(\gamma)}\big) +4\sqrt{\veps} \big(|t_\veps-s_\veps|^2\!+\!\|\rho_\veps\!-\!\chi_\veps \|_{H_1^2(\gamma)}^2\big)+4\veps.
 \end{split}
 \end{equation}
 By \eqref{h-5}, for $\tilde \alpha>4\sqrt{\veps}$,
 \begin{align*}
 &\frac1\theta\big(F(\rho_\veps-\chi_\veps)\!+\!|t_\veps\!-\! s_\veps|^2\big)=\frac1\theta \big(|t_\veps-s_\veps|^2+ \|\rho_\veps-\chi_\veps\|_{H_1^2(\gamma)}^2\big)\\
 &\leq 2C\sqrt{T}+2\sqrt{2M_1}C(\tilde \alpha-4\sqrt{\veps})^{-\frac 12}+8\sqrt{\veps}(\tilde \alpha- 4\sqrt{\veps})^{-1}M_1+4\sqrt{\veps}T^2+4\veps.  \end{align*}
 This implies that
 \begin{equation}\label{h-8}
 \lim_{\theta\to 0} |t_\veps-s_\veps|^2+ \|\rho_\veps-\chi_\veps\|_{H_1^2(\gamma)}^2   =0.
 \end{equation}
 Inserting \eqref{h-5}, \eqref{h-8} into \eqref{h-7}, we obtain that
 \begin{equation}\label{h-9}
 \lim_{\theta\to 0} \frac1\theta\big(\|\rho_\veps-\chi_\veps\|_{H_1^2(\gamma)}^2\!+\! |t_\veps-s_\veps|^2\big) \leq 4\veps.
 \end{equation}

 \noindent\textbf{Case 1}. If there exist a sequence $\theta \to 0$ such that the corresponding $t_\veps\vee s_\veps =T$. Because
 \[\Phi(\bar t,\bar t,\bar\rho,\bar\rho)-\sqrt{\veps}\Delta(\bar t,\bar t,\bar\rho,\bar \rho)\leq \Phi(t_\veps,s_\veps,\rho_\veps, \chi_\veps)-\sqrt{\veps} \Delta(t_\veps,s_\veps,\rho_\veps, \chi_\veps),\]
 we get
 \begin{align*}
   &W(\bar t,\bar \rho)-V(\bar t,\bar \rho)\\
   &\leq W(t_\veps,\rho_\veps)-V(s_\veps,\chi_\veps)+2\beta (T-\bar t)+2\tilde \alpha\e^{2\eta(T-\bar t)} \|\bar\rho\|_{H_1^2(\gamma)}^2 +\sqrt{\veps} \Delta (\bar t,\bar t,\bar \rho,\bar \rho)\\
   &\leq W(t_\veps,\rho_\veps)\!-\! W(t_\veps\!\vee\! s_\veps,\rho_\veps)+ W(t_\veps\!\vee\! s_\veps, \rho_\veps)\!-\! V(t_\veps\!\vee \!s_\veps,\chi_\veps) \! +\! V(t_\veps\!\vee\! s_\veps,\chi_\veps) \! -\! V(s_\veps,\chi_\veps)\\
   &\quad +2\beta (T-\bar t)+\frac{2\lambda}{\bar t}+2\tilde \alpha\e^{2\eta(T-\bar t)} \|\bar\rho\|_{H_1^2(\gamma)}^2 \\
   &\quad +2\sqrt{\veps}(2T^2\!+\! \|\rho_\veps-\bar \rho\|_{H_1^2(\gamma)}^2+\|\chi_\veps -\bar\chi\|_{H_1^2(\gamma)}^2\big)+4\veps\\
   &\leq C (|t_\veps\! -T|^{\frac 12}\!+\!|s_\veps\! -T|^{\frac 12})\!+\!W(T,\rho_\veps)\!-\!V(T,\chi_\veps) +2\beta (T-\bar t)+\frac{2\lambda}{\bar t} \\
    &\quad +\!2\tilde \alpha\e^{2\eta(T-\bar t)} \|\bar\rho\|_{H_1^2(\gamma)}^2 \!+\!4\sqrt{\veps} T^2\!+\!\frac{4 M_1\sqrt{\veps}}{\tilde \alpha \!-\!4\sqrt{\veps}}\!+\!4\sqrt{\veps} \big( \|\bar \rho\|_{H_1^2(\gamma)}^2\!+\!\|\bar\chi\|_{H_1^2(\gamma)}^2\big) \!+ \!4\veps,
 \end{align*}where in the last step we have used \eqref{h-5} for $\tilde \alpha>4\sqrt{\veps}$. Letting first $\theta\to 0$ and then $\veps,\lambda,\beta\to 0$ by taking $\tilde \alpha=4\sqrt{\veps}+\veps^{1/4}$, the previous estimate yields that
 \[W(\bar t,\bar \rho)-V(\bar t,\bar\rho)\leq 0,\]
 which contradicts \eqref{h-2}.

 \noindent\textbf{Case 2}. For any $\theta\in (0,1)$, the corresponding maximum points satisfy $t_\veps\!\vee\!s_\veps<T$.
 Let
 \begin{align*}
 \psi(t,\rho)&=V(s_\veps,\chi_\veps)+\beta (2T-t-s_\veps)+\frac{\lambda}{t}+\frac{\lambda}{s_\veps} +\tilde \alpha\e^{\eta(2T-t-s_\veps)}\big(\|\rho\|_{H_1^2(\gamma)}^2+ \|\chi_\veps\|_{H_1^2(\gamma)}^2\big)\\
 &+\frac1{2\theta} \big(F(\rho-\chi_\veps)+|t-s_\veps|^2\big)+\sqrt{\veps} \Delta(t,s_\veps,\rho,\chi_\veps),
 \end{align*}
 then it follows from \eqref{h-3} that
 \[(t,\rho)\mapsto W(t,\rho)-\psi(t,\rho)=\Phi(t,s_\veps,\rho,\chi_\veps) -\sqrt{\veps}\Delta(t,s_\veps,\rho,\chi_\veps)
 \]
 arrives its maximum at point $(t_\veps,\rho_\veps)$.
 As $W$ is a viscosity subsolution to   \eqref{f-7},
 \begin{equation}\label{h-10}
 -\partial_t \psi(t_\veps,\rho_\veps)-\inf_{\alpha\in \pb(U)}\!\mathcal{H}\Big(t_\veps,\rho_\veps,\frac{\delta \psi(t_\veps,\rho_\veps)}{\delta \rho(x)}, \alpha\Big)\leq 0.
 \end{equation}

 Let
 \begin{align*}
   \tilde \psi(s,\chi)&=W(t_\veps,\rho_\veps)\!-\!\beta(2T\! -\!t_\veps \!-\!s)\!-\!\frac{\lambda}{t_\veps} \!-\! \frac{\lambda}{s} \!-\! \tilde \alpha \e^{\eta(2T-t_\veps-s)}\big( \|\rho_\veps\|_{H_1^2(\gamma)}^2\!+ \! \|\chi\|_{H_1^2(\gamma)}^2\big)\\
   &\quad -\frac1{2\theta} \big( F(\rho_\veps\!-\chi)+|t_\veps\!-s|^2\big)-\sqrt{\veps} \Delta(t_\veps,s,\rho_\veps,\chi),
 \end{align*}
 then
 \[V(s,\chi)-\tilde \psi(s,\chi)=-\Phi(t_\veps,s,\rho_\veps,\chi)\!+\!\sqrt{\veps}\Delta (t_\veps,s, \rho_\veps, \chi)\]
 arrives at its minimum at $(s_\veps,\chi_\veps)$. Hence, it follows from the fact $V$ is a viscosity supersolution to \eqref{f-7} that
 \begin{equation}\label{h-11}
 -\partial_s \tilde \psi(s_\veps, \chi_\veps) -\!\inf_{\alpha\in \pb(U)}\!\mathcal{H}\Big( s_\veps, \chi_\veps, \frac{\delta \tilde \psi(s_\veps,\chi_\veps)}{\delta \chi(x)}, \alpha\Big)\geq 0.
 \end{equation}

 According to Lemma \ref{lem-8},
 \begin{align*}
   \frac{\delta \psi(t_\veps,\rho_\veps)}{\delta \rho(x)}&=\frac 1\theta (\rho_\veps\!-\!\chi_\veps)(x)\gamma(x) \!+ \! 2\tilde \alpha\e^{\eta(2T-t_\veps-s_\veps)} \rho_\veps(x)\gamma(x) \! +\! 2\sqrt{\veps}\sum_{k=1}^\infty \beta_k(\rho_\veps\!- \!\rho_k )(x)\gamma(x),\\
   \frac{\delta\tilde \psi(s_\veps,\chi_\veps)}{\delta \chi(x)}&=\frac 1\theta (\rho_\veps\!-\! \chi_\veps)(x) \gamma(x)\! -\! 2\tilde \alpha \e^{\eta(2T-t_\veps-s_\veps)} \chi_\veps(x)\gamma(x)\!-\! 2\sqrt{\veps}\sum_{k=1}^\infty \beta_k (\chi_\veps\!-\!\chi_k)(x)\gamma(x).
 \end{align*}
  We deduce from \eqref{h-10} and \eqref{h-11} that
 \begin{equation*}
 \begin{split}
   &\partial_t\psi(t_\veps,\rho_\veps)-\partial_s\tilde \psi(s_\veps,\chi_\veps)\\
   &=-2\beta\!-\!\frac{\lambda}{t_\veps^2} \! -\! \frac{\lambda}{s_\veps^2}\!
   -\!2\tilde \alpha \eta\e^{\eta(2T-t_\veps-s_\veps)} \big(\|\rho_\veps\|_{H_1^2(\gamma)}^2\!+\! \|\chi_\veps\|_{H_1^2(\gamma)} \big)\!-\! 2\sqrt{\veps} \sum_{k=1}^\infty\! \beta_k(t_\veps\!-t_k\! + \! s_\veps\!-\!s_k)\\
   &\geq \inf_{\alpha\in \pb(U)}\!\Big\{\mathcal{H}\Big(s_\veps, \chi_\veps, \frac{\delta \tilde \psi(s_\veps,\chi_\veps)}{\delta \chi(x)},\alpha\Big)-\mathcal{H}\Big( t_\veps, \rho_\veps, \frac{\delta \psi(t_\veps,\chi_\veps)}{\delta \rho(x)}, \alpha\Big)\Big\}
   \end{split}
   \end{equation*}
   \begin{equation}\label{h-12}
   \begin{split}
   &=\inf_{\alpha\in\pb(U)}\!\Big\{ \int_{\R^d} \! \big \la b(s_\veps,x,\chi_\veps,\alpha), D\Big(\frac{\delta \tilde \psi(s_\veps,\chi_\veps)}{\delta \chi(x)}\Big) \big\raa \chi_\veps(x)\d x \\
   &\qquad \qquad \quad -  \int_{\R^d} \! \big \la b(t_\veps,x,\rho_\veps,\alpha), D\Big(\frac{\delta \psi(t_\veps,\rho_\veps)}{\delta \rho(x)}\Big) \big\raa \rho_\veps(x)\d x\\
   &\qquad -\frac 12\int_{\R^d}\!\big\la AD\Big(\frac{\delta \tilde \psi(s_\veps,\chi_\veps)}{\delta \chi(x)}\Big),D\chi_\veps(x)\big \raa \d x+\frac 12 \int_{\R^d}\! \big\la A D\Big(\frac{\delta \psi(t_\veps,\rho_\veps)}{\delta \rho(x)}\Big), D\rho_\veps(x) \big\raa \d x\\
   &\qquad +\int_{\R^d}\! f(s_\veps, x,\chi_\veps,\alpha)\chi_\veps(x)\d x-\int_{\R^d}\! f(t_\veps, x,\rho_\veps,\alpha)\rho_\veps(x)\d x\Big\}\\
   &=:\inf_{\alpha\in \pb(U)}\big\{ \mathrm{(I)}+\mathrm{(I\!I)}+\mathrm{(I\!I\!I)} \big\}.
 \end{split}
 \end{equation}
 We shall estimate the terms $\mathrm{(I)}$, $\mathrm{(I\!I)}$, $\mathrm{(I\!I\!I)}$ one by one below.

Let us first deal with term $(\mathrm{I\!I\!I})$. Using $(\mathbf{A}_3)$ and \eqref{h-5},
\begin{equation}\label{h-13}
\begin{aligned}
   (\mathrm{I\!I\!I}) &\geq -\int_{\R^d}\! |f(s_\veps,x,\chi_\veps,\alpha)-f(t_\veps,x,\rho_\veps,\alpha)|
  \chi_\veps(x)\d x\\
  &\quad -\int_{\R^d}\! |f(t_\veps, x,\rho_\veps,\alpha) | |\rho_\veps(x)-\chi_\veps(x)|\d x\\
  &\geq -\tilde K_3\big(|t_\veps\!-\!s_\veps|\! +\! \|\rho_\veps \! -\! \chi_\veps\|_{L^2(\gamma)}\big)\int_{\R^d}\!  \chi_\veps(x)\d x \!- \!\tilde K_3\int_{\R^d}\! |\rho_\veps(x) \!-\!\chi_\veps(x)| \d x\\
  &\geq -\tilde K_3C_1\big(|t_\veps-s_\veps|+\|\rho_\veps -\chi_\veps\|_{L^2(\gamma)}\big) \|\chi_\veps\|_{L^2(\gamma)}-\tilde K_3C_1 \|\rho_\veps-\chi_\veps\|_{L^2(\gamma)},
\end{aligned}
\end{equation}
where $C_1=\big(\int_{\R^d}\!\gamma(x)^{-1}\d x\big)^{\frac 12}\in (0,\infty)$.

Next, we go to estimate term $(\mathrm{I\!I})$. Recall that $A=(a_{ij})$ denotes the matrix $\sigma\sigma^\ast$.
\begin{align}\notag
  (\mathrm{I\!I})&=\frac 12\int_{\R^d}\!\big[\la A D\Big(\frac{\delta \psi(t_\veps,\rho_\veps)} {\delta \rho(x)}\Big), D\rho_\veps (x)\raa -\big\la A D\Big(\frac{\delta \tilde \psi(s_\veps,\chi_\veps)} {\delta \chi (x)}\Big), D\chi_\veps(x)\big \raa\big]\d x\\ \notag
  &=\frac 1{2\theta} \int_{\R^d}\! \la A D\big((\rho_\veps-\chi_\veps)(x)\gamma(x)\big), D\rho_\veps(x)-D\chi_\veps(x)\raa \d x\\ \notag
  &\quad+\tilde \alpha \e^{\eta (2T-t_\veps-s_\veps)}\int_{\R^d}\! \big[\la A D\big(\rho_\veps (x)\gamma(x)\big),D\rho_\veps(x)\raa \!+\!\la A D\big(\chi_\veps(x)\gamma(x)\big), D\chi_\veps(x)\raa \big] \d x\\ \notag
  &\quad +\!\sqrt{\veps}\sum_{k=1}^\infty\! \beta_k\!\!\int_{\R^d}\! \!\big[ \la A D\big((\rho_\veps\!-\! \rho_k)(x)\gamma(x)\big), D\rho_\veps(x)\raa\!+\!\la A D\big((\chi_\veps\!-\! \chi_k)(x) \gamma(x)\big), D\chi_\veps(x)\raa \big] \d x\\ \notag 
  &=:(\mathrm{I\!I}_1)+(\mathrm{I\!I}_2)+(\mathrm{I\!I}_3).
\end{align}
Using \eqref{norm-1} and the positive definiteness of matrix $A$ under ($\mathbf{A}_2$),
\begin{equation}\label{h-15}
\begin{split}
  (\mathrm{I\!I}_1)&=\frac{1}{\theta} \int_{\R^d}\! \la A D(\rho_\veps-\chi_\veps)(x), D(\rho_\veps-\chi_\veps) (x) \raa \gamma(x)\d x\\
  &\quad +\frac{1}{4\theta}\int_{\R^d}\! \la A D\gamma(x), D\big((\rho_\veps(x)-\chi_\veps(x))^2\big)\raa \d x\\
  &\geq \frac{1}{4\theta}\int_{\R^d}\! \la A D\gamma(x), D\big((\rho_\veps(x)-\chi_\veps(x))^2\big)\raa \d x\\
  &=-\frac 1{4\theta} \int_{\R^d} \sum_{i,j=1}^d a_{ij}\partial_{x_ix_j}^2 \gamma(x)(\rho_\veps(x)-\chi_\veps(x))^2\d x\\
  &\geq-   \frac{\kappa \sum_{i,j=1}^d |a_{ij}|}{4}\frac{1}{\theta} \|\rho_\veps-\chi_\veps\|_{L^2(\gamma)}^2,
\end{split}
\end{equation}  where in the last step \eqref{norm-1} has been used.
Since $A$ is positive definite, using the integration by parts formula and \eqref{norm-1},
\begin{equation}\label{h-16}
\begin{split}
  (\mathrm{I\!I}_2)&\geq \tilde \alpha\e^{\eta (2T-t_\veps -s_\veps)} \int_{\R^d}\!\frac 12 \la A D\gamma(x), D\rho_\veps^2(x)+D\chi_\veps^2(x)\raa \d x\\
  &\geq -\frac{\kappa\sum_{i,j=1}^d |a_{ij}|}{2} \tilde \alpha\e^{\eta (2T-t_\veps -s_\veps)}\big(\|\rho_\veps\|_{L^2(\gamma)}^2+\| \chi_\veps\|_{L^2(\gamma)}^2\big).
\end{split}
\end{equation}   Similarly, we can estimate $(\mathrm{I\!I}_3)$ as follows. Using \eqref{norm-1}  and \eqref{h-3.5},
\begin{equation*}
\begin{split}
  (\mathrm{I\!I}_3)&=\sqrt{\veps}\sum_{k=1}^\infty \beta_k\! \int_{\R^d}\! \!\big[\la A D(\rho_\veps\!-\!\rho_k)(x), D\rho_\veps(x)\raa \gamma(x)\!+\!\la A D\gamma(x), D\rho_\veps(x)\raa (\rho_\veps(x)\!-\! \rho_k(x))\big]\d x\\
  & +\sqrt{\veps}\sum_{k=1}^\infty \beta_k\!\int_{\R^d} \!\!\big[\la A D(\chi_\veps\!-\!\chi_k)(x), D\chi_\veps(x)\raa \gamma(x)\!+\!\la A D\gamma(x), D\chi_\veps(x)\raa (\chi_\veps(x)\!-\! \chi_k(x))\big]\d x\\
  &\geq -\sqrt{\veps}\kappa\sum_{k=1}^\infty \beta_k\|A\|\big( \|D\rho_\veps\! -\! D\rho_k\|_{L^2(\gamma)} \|D\rho_\veps\|_{L^2(\gamma)}\! +\!
   \|D\rho_\veps\|_{L^2(\gamma)}\|\rho_\veps \! -\! \rho_k\|_{L^2(\gamma)}\big)\\
   &\quad -\sqrt{\veps}\kappa\sum_{k=1}^\infty \beta_k\|A\|\big( \|D\chi_\veps\! -\! D\chi_k\|_{L^2(\gamma)} \|D\chi_\veps\|_{L^2(\gamma)}\! +\!
   \|D\chi_\veps\|_{L^2(\gamma)}\|\chi_\veps \! -\! \chi_k\|_{L^2(\gamma)}\big)\\
   &\geq -  2\kappa \veps^{\frac 34} \|A\|\big( \|D\rho_\veps\|_{L^2(\gamma)}+\|D\chi_\veps\|_{L^2(\gamma)} \big),
\end{split}
\end{equation*} where $\|A\|=\sup\{|A\xi|; \xi\in \R^d, |\xi|=1\}$. Combining this estimate with \eqref{h-16}, \eqref{h-15}, we  finally get that
\begin{equation}\label{h-17}
\begin{split}
(\mathrm{I\!I})&\geq -   \frac{C_2 }{\theta} \|\rho_\veps-\chi_\veps\|_{L^2(\gamma)}^2\! - \! 2\kappa \veps^{\frac 34} \|A\| \big( \|D\rho_\veps\|_{L^2(\gamma)} \!+\!  \|D\chi_\veps\|_{L^2(\gamma)} \big) \\
&\quad -2C_2 \tilde \alpha\e^{\eta (2T-t_\veps -s_\veps)}\big(\|\rho_\veps\|_{L^2(\gamma)}^2+\| \chi_\veps\|_{L^2(\gamma)}^2\big),
\end{split}
\end{equation}
where $C_2= {\kappa \sum_{i,j=1}^d |a_{ij}|}/{4}$.

At last, we deal with term $(\mathrm{I})$.
\begin{equation}\label{h-18}
\begin{split}
  (\mathrm{I})&=\frac1\theta\int_{\R^d}\!  \la b(s_\veps, x,\chi_\veps,\alpha), D\big((\rho_\veps(x)-\chi_\veps(x))\gamma(x)\big)\raa \chi_\veps(x) \d x\\
  &\quad -\frac 1\theta\int_{\R^d}\! \la b(t_\veps, x,\rho_\veps, \alpha), D\big((\rho_\veps(x)-\chi_\veps(x))\gamma(x)\big)\raa \rho_\veps(x)\d x\\
  &-\tilde \alpha \e^{\eta(2T- t_\veps - s_\veps)} \int_{\R^d}\!\la b(s_\veps, x,\chi_\veps,\alpha), D\big(\chi_\veps(x)\gamma(x)\big)\raa \chi_\veps(x) \d x\\
  &-\tilde \alpha \e^{\eta(2T- t_\veps - s_\veps)} \int_{\R^d}\!\la b(t_\veps,x,\rho_\veps,\alpha), D\big(\rho_\veps(x)\gamma(x)\big)\raa \rho_\veps(x)\d x\\
  &-2\sqrt{\veps}\!\sum_{k=1}^\infty \!\beta_k\!\int_{\R^d}\! \la b(s_\veps,x,\chi_\veps,\alpha), D\big( (\chi_\veps(x)-\chi_k(x)) \gamma(x)\big)\raa \chi_\veps(x)\d x\\
  &-2\sqrt{\veps}\!\sum_{k=1}^\infty \!\beta_k\!\int_{\R^d}\! \la b(t_\veps,x,\rho_\veps,\alpha), D\big((\rho_\veps(x)-\rho_k(x))\gamma(x)\big)\raa \rho_\veps (x)\d x\\
  &=:(\mathrm{I}_1)+(\mathrm{I}_2)+(\mathrm{I}_3).
\end{split}
\end{equation}
Similar to the estimates of $(\mathrm{I\!I}_2)$ and $(\mathrm{I\!I}_3)$, we can obtain that
\begin{align}\label{h-20}
(\mathrm{I}_2)&\geq -\tilde \alpha \e^{\eta(2T- t_\veps - s_\veps)} \tilde K_2 (\kappa+\frac 12)\big(\|\rho_\veps\|_{H_1^2(\gamma)}^2 +\|\chi_\veps\|_{H_1^2(\gamma)}^2\big),\\ \notag
(\mathrm{I}_3)&\geq - 4\tilde K_2\kappa\sqrt{\veps}\sum_{k=1}^\infty \beta_k\big(\|\chi_\veps-\chi_k\|_{H_1^2(\gamma)} \|\chi_\veps\|_{L^2(\gamma)} + \|\rho_\veps-\rho_k\|_{H_1^2(\gamma)}  \|\rho_\veps\|_{L^2(\gamma)}\big) \\ \label{h-21}
&\geq -4\tilde K_2\kappa\veps^{\frac 34}\big(\|\chi_\veps\|_{L^2(\gamma)} +\|\rho_\veps\|_{L^2(\gamma)}\big).
\end{align}
Furthermore, using $(\mathbf{A}_1)$, \eqref{norm-1}  and H\"older's inequality,
\begin{equation}\label{h-19}
\begin{aligned}
  (\mathrm{I}_1)&=\frac1\theta \int_{\R^d} \!\la b(s_\veps, x,\chi_\veps,\alpha)\!-\! b(t_\veps,x,\rho_\veps, \alpha), D\rho_\veps(x)\!-\!D\chi_\veps(x)\raa \chi_\veps(x)\gamma(x)\d x\\
  &\quad +\frac 1\theta \int_{\R^d}\!\la b(t_\veps,x,\rho_\veps,\alpha),D\rho_\veps(x) -D\chi_\veps(x)\raa (\chi_\veps(x)-\rho_\veps(x)) \gamma(x)\d x\\
  &\quad + \frac 1\theta \int_{\R^d}\! \la b(s_\veps, x,\chi_\veps,\alpha)\!-\! b(t_\veps,x,\rho_\veps, \alpha), D\gamma(x)\raa (\rho_\veps(x)\! -\!\chi_\veps (x))\chi_\veps(x)\d x\\
  &\quad -\frac 1\theta\int_{\R^d}\! \la b(t_\veps,x,\rho_\veps,\alpha), D\gamma(x)\raa \big(\rho_\veps(x)-\chi_\veps(x)\big)^2\d x\\
  &\geq -\frac{\tilde K_1}\theta  \big(| s_\veps-t_\veps| + \| \rho_\veps-\chi_\veps \|_{L^2(\gamma)}\big) \int_{\R^d}\! |D\rho_\veps(x)\!-\!D\chi_\veps(x)| \chi_\veps(x)\gamma(x)\d x\\
  &\quad -\frac{\tilde K_2}{\theta} \int_{\R^d}\!|D\rho_\veps(x)-D\chi_\veps(x)||\rho_\veps(x) -\chi_\veps(x)|\gamma(x)\d x\\
  &\quad -\frac{\tilde K_1\kappa}\theta  \big(| s_\veps-t_\veps| + \| \rho_\veps-\chi_\veps \|_{L^2(\gamma)}\big)\int_{\R^d}\! |\rho_\veps (x)-\chi_\veps(x)|\chi_\veps(x)\gamma(x)\d x\\
  &\quad -\frac{\tilde K_2\kappa}{\theta}\int_{\R^d} \big(\rho_\veps(x)-\chi_\veps(x)\big)^2\gamma(x)\d x\\
  &\geq -\frac{\tilde K_1\kappa}\theta  \big(| s_\veps-t_\veps| + \| \rho_\veps-\chi_\veps \|_{L^2(\gamma)}\big)\|\rho_\veps-\chi_\veps\|_{H_1^2(\gamma)} \|\chi_\veps\|_{L^2(\gamma)}\\
  &\quad -\frac{\tilde K_2}{\theta} \|D\rho_\veps-D\chi_\veps\|_{L^2(\gamma)}\|\rho_\veps -\chi_\veps\|_{L^2(\gamma)} -\frac{\tilde K_2\kappa}{\theta} \|\rho_\veps-\chi_\veps\|_{L^2(\gamma)}^2\\
  &\geq -\frac{2\tilde K_1\kappa}{\theta}\big(|s_\veps\! -\! t_\veps|^2\!+ \! \|\rho_\veps\!-\!  \chi_\veps\|_{H_1^2(\gamma)}^2\big) \|\chi_\veps\|_{L^2(\gamma)} -\frac{2\tilde K_2\kappa}{\theta} \|\rho_\veps \!-\! \chi_\veps\|_{H_1^2(\gamma)}^2.
\end{aligned}
\end{equation}

Consequently, inserting \eqref{h-13}, \eqref{h-17}, \eqref{h-19}-\eqref{h-21} into \eqref{h-12}, we obtain
\begin{align*}
  &-2\beta -2\sqrt{\veps} \sum_{k=1}^\infty\beta_k(t_\veps-t_k+s_\veps-s_k)\\
  &\geq  \tilde \alpha \e^{\eta(2T-t_\veps-s_\veps)}\big(2\eta -2C_2-\tilde K_2(\kappa+\frac 12)\big) \big(\|\rho_\veps\|_{H_1^2(\gamma)}^2 +\|\chi_\veps\|_{H_1^2(\gamma)}^2\big) \\
  &\quad -\tilde K_3C_1\big(|t_\veps-s_\veps|+\|\rho_\veps -\chi_\veps\|_{L^2(\gamma)}\big) \|\chi_\veps\|_{L^2(\gamma)}  -\tilde K_3 C_1 \|\rho_\veps -\chi_\veps\|_{L^2(\gamma)}\\
  &\quad - \frac{C_2 }{\theta} \|\rho_\veps- \chi_\veps\|_{L^2(\gamma)}^2-2K_3\veps^{\frac 34}\|A\|\big(\|D\rho_\veps\|_{L^2(\gamma)} \!+\! \|D\chi_\veps\|_{L^2(\gamma)} \big)\\
  &\quad -4\tilde K_2\kappa\veps^{\frac 34} \big(\|\chi_\veps\|_{L^2(\gamma)}+ \|\rho_\veps\|_{L^2(\gamma)}\big)\\
  &\quad -\frac{2\tilde K_1\kappa}{\theta}\big(|s_\veps -t_\veps|^2\!+\!\|\rho_\veps\!-\! \chi_\veps\|_{H_1^2(\gamma)}^2\big) \|\chi_\veps\|_{L^2(\gamma)} -\frac{2\tilde K_2\kappa}{\theta}\|\rho_\veps -\chi_\veps\|_{H_1^2(\gamma)}^2.
\end{align*}
We take $\eta>0$ sufficiently large so that $2\eta -2C_2-\tilde K_2(\kappa+\frac 12)\geq 0$, and use the estimate \eqref{h-5} to deduce from the previous inequality that
\begin{align*}
  &-2\beta -2\sqrt{\veps} \sum_{k=1}^\infty\beta_k(t_\veps-t_k+s_\veps-s_k)\\
  &\geq -\tilde K_3C_1\big( |t_\veps-s_\veps|+\|\rho_\veps -\chi_\veps\|_{L^2(\gamma)}\big) \Big(\frac{M_1}{\tilde \alpha-4\sqrt{\veps}}\Big)^{\frac 12} -\tilde K_3 C_1\|\rho_\veps -\chi_\veps\|_{L^2(\gamma)}\\
  &\quad -\frac{C_2}{\theta} \|\rho_\veps-\chi_\veps\|_{L^2(\gamma)}^2-\big( 4K_3\|A\|+ 8\tilde K_2\kappa\big)\veps^{\frac 34} \Big(\frac{M_1}{\tilde \alpha-4\sqrt{\veps}}\Big)^{\frac 12}\\
  &\quad -\!\frac{2\tilde K_1\kappa}{\theta} \big(|s_\veps\! -\! t_\veps|^2\!+\!\|\rho_\veps\!-\! \chi_\veps\|_{H_1^2(\gamma)}^2\big) \Big(\frac{M_1}{\tilde \alpha-4\sqrt{\veps}}\Big)^{\frac 12}\! -\! \frac{2\tilde K_2\kappa}{\theta}\|\rho_\veps\! - \!\chi_\veps\|_{H_1^2(\gamma)}^2.
\end{align*}
Letting first $\theta\to 0$  then $\veps\to 0$, according to \eqref{h-8}, \eqref{h-9}, we obtain that
\[-2\beta\geq 0,\]
which contradicts the condition $\beta\in (0,1)$. Thus, we  complete  the proof of this theorem.
\end{proof}

As an application of the comparison principle, in view of Theorem \ref{thm-3},  we can obtain the desired  uniqueness theorem for the value function, whose proof is standard and hence is omitted.

\begin{mycor}[Uniqueness of viscosity solution]
Assume $(\mathbf{A}_1)$-$(\mathbf{A}_3)$  hold. Then the value function $V$ to the optimal control problem \eqref{a-4} is the unique viscosity solution to the  HJB equation \eqref{f-7} on $[0,T]\times \pb_1^r(\R^d)$  satisfying \eqref{cont}.
\end{mycor}

\end{document}